\documentclass[twoside,a4paper,reqno,11pt]{amsart} 
\usepackage{amsfonts, amsbsy, amsmath, amssymb, latexsym}
\usepackage{mathrsfs,array}
\usepackage[top=30mm,right=30mm,bottom=30mm,left=30mm]{geometry}
\usepackage{stmaryrd}
\usepackage{bm}
\usepackage{xcolor}
\usepackage{hyperref}
\usepackage{listings}
\usepackage{tikz}
\usetikzlibrary{graphs}
\usepackage{longtable}
\usepackage{lscape}

\allowdisplaybreaks

\usepackage{booktabs}

\newtheorem{theorem}{Theorem}[section]

\newtheorem{lemma}[theorem]{Lemma}

{\theoremstyle{definition}
	\newtheorem{definition}[theorem]{Definition}
	\newtheorem{remark}[theorem]{Remark}

}

{\theoremstyle{definition}

}

\usepackage{mathtools}

\begin{document}
	
\author{Melissa Lee and Tomasz Popiel}
\address{School of Mathematics, Monash University, Clayton VIC 3800, Australia}
\email{melissa.lee@monash.edu, tomasz.popiel@monash.edu}

\title[Recognisability of the sporadic groups]{Recognisability of the sporadic groups by the isomorphism types of their prime graphs}

\begin{abstract} 
The {\em prime graph} of a finite group $G$ is the labelled graph $\Gamma(G)$ with vertices the prime divisors of $|G|$ and edges the pairs $\{p,q\}$ for which $G$ contains an element of order $pq$. 
A group $G$ is {\em recognisable} by its prime graph if every group $H$ with $\Gamma(H)=\Gamma(G)$ is isomorphic to $G$. 
Cameron and Maslova have shown that every group that is recognisable by its prime graph is almost simple, which justifies the significant amount of attention that has been given to determining which simple (or almost simple) groups are recognisable by their prime graphs. 
This problem has been completely solved for certain families of simple groups, including the sporadic groups. 
A natural extension of the problem is to determine which groups are recognisable by their {\em unlabelled} prime graphs, i.e. by the isomorphism types of their prime graphs. 
There seem to be only very limited results in this direction in the literature. 
Here we determine which of the sporadic finite simple groups are recognisable by the isomorphism types of their prime graphs. 
We also show that for every sporadic group $G$ that is not recognisable by the isomorphism type of $\Gamma(G)$, there are infinitely many groups $H$ with $\Gamma(H) \cong \Gamma(G)$. 
\end{abstract}

\thanks{The first author acknowledges the support of an Australian Research Council Discovery Early Career Researcher Award, project no.~DE230100579. 
We thank an anonymous referee for several helpful comments.}

\date{\today}
\dedicatory{Dedicated to Cheryl Praeger on the occasion of her 75th birthday}
\maketitle

\section{Introduction}	 \label{s:intro}

All groups considered in this paper are finite.

The {\em prime} graph of a group $G$ is the labelled graph $\Gamma(G)$ with vertices the prime divisors of $|G|$ and edges the pairs $\{p,q\}$ for which $G$ contains an element of order $pq$. 
It is also called the {\em Gruenberg--Kegel graph}, having been introduced by Gruenberg and Kegel in an unpublished manuscript in 1975. 
We use the term {\em prime graph} here for brevity. 
A group $G$ is {\em recognisable} by its prime graph if every group $H$ with $\Gamma(H)=\Gamma(G)$ is isomorphic to $G$. 
More generally, $G$ is $k$-{\em recognisable} by its prime graph if there are precisely $k$ pairwise non-isomorphic groups with the same prime graph as $G$. 
If $G$ is not $k$-recognisable by its prime graph for any $k$, then $G$ is said to be {\em unrecognisable} by its prime graph. 
(Note that the terms ``unrecognisable'' and ``not recognisable'' are therefore not synonymous.)

The question of recognisability of various groups by their prime graphs has attracted significant interest. 
We refer the reader to the recent article of Cameron and Maslova~\cite{CamMas} for an up-to-date literature review (and several significant results). 
Much attention has been given to simple groups, and this is justified by \cite[Theorem~1.2]{CamMas}, which shows that every group $G$ that is $k$-recognisable by its prime graph for some $k$ is almost simple, i.e. $G_0 \trianglelefteq G \leq \text{Aut}(G_0)$ for some non-abelian simple group $G_0$. 
The recognisability problem remains open for many families of simple groups, but has been completely settled in some cases.  
In particular, the case of the $26$ sporadic simple groups was recently completed by the present authors \cite{MBCo1}, building on a significant body of earlier work due in particular to Hagie~\cite{Hagie} and Kondrat'ev~\cite{Kon1,Kon2}; see Table~\ref{tabSummary}.
Another major recent result is due to Maslova, Panshin, and Staroletov~\cite{NMexceptional}, who have shown that all simple exceptional groups of Lie type apart from $\text{G}_2(3)$ and the Suzuki groups are $k$-recognisable for some $k$.

A natural extension of the recognisability question is the following. 
Given a group $G$ that is recognisable by its prime graph $\Gamma(G)$, is $G$ still recognisable if we remove the vertex labelling from $\Gamma(G)$? 
In other words, is it the case that every group $H$ with prime graph $\Gamma(H)$ {\em isomorphic}, but not necessarily equal, to $\Gamma(G)$, is isomorphic to $G$? 
If the answer is yes, then $G$ is said to be {\em recognisable by the isomorphism type} of its prime graph. 
This question has been raised (in particular) by Cameron and Maslova, who showed that the simple groups ${}^2\text{G}_2(27)$ and $\text{E}_8(2)$ are recognisable by the isomorphism types of their prime graphs \cite[Theorem~1.5]{CamMas}. 
We have also been informed that Maslova and Zinonieva have shown that the groups $\text{E}_8(q)$ for $q \in \{3,4,5,7,8,9,17\}$ are recognisable by the isomorphism types of their prime graphs. 
Beyond this, not much seems to be known. 

The primary purpose of this paper is to determine which of the $26$ sporadic finite simple groups are recognisable by the isomorphism types of their prime graphs. 
Of course, given a group $G$ (sporadic or otherwise) that is {\em not} recognisable by the isomorphism type of $\Gamma(G)$, it is natural to ask whether $G$ is {\em $k$-recognisable} in the sense that there are exactly~$k$ (pairwise non-isomorphic) groups $H$ with $\Gamma(H) \cong \Gamma(G)$ for some $k$, or {\em unrecognisable} in the sense that there are infinitely many groups $H$ with $\Gamma(H) \cong \Gamma(G)$. 
We also completely answer this secondary question. 
Our main result is as follows.

\begin{theorem} \label{thm1}
A sporadic finite simple group $G$ is recognisable by the isomorphism type of its prime graph if and only $G$ is one of the following eight groups:
\[
\mathbb{B},\; \textnormal{Fi}_{23},\; \textnormal{Fi}_{24}',\; \textnormal{J}_4,\; \textnormal{Ly},\; \mathbb{M},\; \textnormal{O'N},\; \textnormal{Th}. 
\]
Moreover, for every sporadic finite simple group $G$ not appearing in the above list, there are infinitely many pairwise non-isomorphic groups $H$ with $\Gamma(H) \cong \Gamma(G)$. 
\end{theorem}

We also point out the following error in the literature. 

\begin{remark} \label{rem:Ru}
It is asserted in \cite{Kon1} that the sporadic group $\text{Ru}$ is recognisable by its labelled prime graph. 
This is incorrect. 
There are infinitely many groups with the same prime graph as $\text{Ru}$. 
A proof is given in Section~\ref{ss:Ru}. 
The revised solution to the recognisability-by-labelled-prime-graph problem for the sporadic groups is summarised in Table~\ref{tabSummary}. 
\end{remark}

\section{Preliminaries} \label{s:prelim}

We first establish some notation, which is mostly standard in the literature on prime graphs. 
Given a group $G$, let $\pi(G)$ denote the set of prime divisors of $|G|$. 
Let $t(G)$ denote the independence number of the prime graph $\Gamma(G)$ of $G$, i.e. the maximal size of a coclique in $\Gamma(G)$. 
Given $p \in \pi(G)$, let $t(p,G)$ denote the maximal size of a coclique in $\Gamma(G)$ containing the vertex $p$. 
Let $s(G)$ denote the number of connected components of $\Gamma(G)$, and label these connected components $\pi_i(G)$ for $i = 1,\ldots,s(G)$. 
If $|G|$ is even, declare $\pi_1(G)$ to be the component containing the vertex $2$.

Let $F(G)$ denote the Fitting subgroup of $G$, i.e. the unique maximal nilpotent subgroup of $G$, and let $R(G)$ denote the soluble radical of $G$, i.e. the unique maximal soluble subgroup of $G$. 
Write $G=N{:}M$ for a split extension with normal subgroup $N \lhd G$ and complement $M < G$ isomorphic to $G/N$. 
The {\em socle} of $G$ is the subgroup of $G$ generated by the minimal normal subgroups of $G$. 
If $G$ is {\em almost simple}, i.e. $S \trianglelefteq G \leq \text{Aut}(S)$ for some non-abelian simple group $S$, then the socle of $G$ is equal to $S$. 

In their unpublished manuscript, Gruenberg and Kegel proved a theorem describing the structure of $G$ in the case in which $\Gamma(G)$ is disconnected. 
This theorem, which is now known as the {\em Gruenberg--Kegel Theorem}, was later published by Gruenberg's student Williams \cite[Theorem~A]{Williams}. 
There are various versions of the Gruenberg--Kegel Theorem appearing in the literature; we present a version based  on \cite[Lemma~2]{Zavar}. 

\begin{theorem}[Gruenberg--Kegel] \label{GKthm}
If $H$ is a finite group with disconnected prime graph, then one of the following statements holds.
\begin{itemize}
\item[(i)] $H$ is a Frobenius group and $\Gamma(H)$ consists of two connected components $\Gamma(K)$ and $\Gamma(C)$, where $K$ and $C$ are the Frobenius kernel and complement of $H$, respectively; $\Gamma(K)$ is a complete graph; and $\Gamma(C)$ is either a complete graph, or $2,3,5 \in \pi(C)$ and $\Gamma(C)$ is obtained from the complete graph on $\pi(C)$ by deleting the edge $\{3,5\}$. 
\item[(ii)] $H$ is a $2$-Frobenius group, i.e. $H$ has shape $(A{:}B){:}C$ and both $A{:}B$ and $B{:}C$ are Frobenius groups. 
In this case, $\Gamma(H)$ consists of two connected components $\Gamma(AC)$ and $\Gamma(B)$, both of which are complete graphs. 
\item[(iii)] There exists a non-abelian simple group $S$ such that $S \trianglelefteq \overline{H} := H/F(H) \leq \textnormal{Aut}(S)$. 
Moreover, $\Gamma(S)$ is disconnected with $s(S) \geq s(H)$, and both $\pi(F(H))$ and $\pi(\overline{H}/S)$ are (possibly empty) subsets of $\pi_1(H)$. 
\end{itemize}
\end{theorem}

\begin{proof}
The version of the Gruenberg--Kegel Theorem proved by Williams~\cite[Theorem~A]{Williams} asserts that either (i)$'$ $H$ is a Frobenius group, (ii)$'$ $H$ is a $2$-Frobenius group, or (iii)$'$ there exists a nilpotent normal subgroup $N$ of $H$ and a non-abelian simple group $S$ such that $S \trianglelefteq H/N \leq \text{Aut}(S)$ and $\pi(N)$ and $\pi((H/N)/S)$ are both subsets of $\pi_1(H)$ (which does indeed contain $2$ because $H$ must have even order, given that it has a non-abelian composition factor \cite{FTthm}). 
In case (iii)$'$, it follows that $H/F(H)$ is also almost simple with socle $S$ because $N \leq F(H)$, so we can replace $N$ by $F(H)$ to obtain case (iii) as stated. 

It remains to justify the claims about the connected components of $\Gamma(H)$ in cases (i)$'$ and (ii)$'$, so that we can replace these by cases (i) and (ii) as stated. 
For a proof of the fact that $\Gamma(H)$ has exactly two connected components (of the stated forms) in these cases, see \cite[Lemma~4.3]{GruenRogg}. 
In case (i)$'$, the Frobenius kernel $K$ of $H$ is nilpotent by a well-known theorem of Thompson~\cite{Thompson}, so $\Gamma(K)$ is complete because $K$ is a direct product of its Sylow subgroups; the description of $\Gamma(H)$ given in (i) follows from Zassenhaus' classification of Frobenius complements \cite{Zass1,Zass2}. 
In case (ii)$'$, the fact that both components of $\Gamma(H)$ are complete graphs follows from \cite[Lemma~1]{Mazurov} and \cite[Lemma~3]{MazurovEtal}.
\end{proof}

We also need the following refinement of the Gruenberg--Kegel Theorem due to Vasil'ev and Gorshkov \cite[Theorem~1]{VasilevGorshkov} (see also \cite{Vasilev}). 

\begin{theorem} \label{GKthmVasilev}
If $H$ is a finite group with $t(H) \geq 3$ and $t(2,H) \geq 2$, then all of the following statements hold. 
\begin{itemize}
\item[(i)] There exists a non-abelian simple group $T$ such that $T \trianglelefteq \hat{H} := H/R(H) \leq \textnormal{Aut}(T)$. 
\item[(ii)] For every coclique $\rho \subseteq \pi(H)$ in $\Gamma(H)$ of size at least $3$, at most one prime in $\rho$ divides $|R(H)|\cdot|\hat{H}/T| = |H|/|T|$. In particular, $t(T) \geq t(H)-1$. 
\item[(iii)] One of the following statements holds.
	\begin{itemize}
	\item[(a)] Every prime $r \in \pi(H)$ that is not adjacent to $2$ in $\Gamma(H)$ does not divide $|R(H)|\cdot|\hat{H}/T| = |H|/|T|$. In particular, $t(2,T) \geq t(2,H)$. 
	\item[(b)] There exists a prime $r \in \pi(R(H))$ that is not adjacent to $2$ in $\Gamma(H)$. In this case, $t(H) = 3$, $t(2,H) = 2$, and $T \cong \textnormal{Alt}_7$ or $\textnormal{PSL}_2(q)$ for some odd $q$. 
	\end{itemize}
\end{itemize}
\end{theorem}

\begin{remark} \label{rem:S=T}
Suppose that $H$ is a group with disconnected prime graph and that we have managed to deduce that $H$ is in case (iii) of Theorem~\ref{GKthm}, so that the quotient of $H$ by its Fitting subgroup $F(H)$ is almost simple with socle $S$ for some non-abelian simple group $S$. 
If we instead take the quotient of $H$ by the potentially larger soluble normal subgroup $R(H)$, then we again obtain an almost simple group with the same socle $S$. 
If $H$ also satisfies the hypotheses of Theorem~\ref{GKthmVasilev}, then it follows that the non-abelian simple group $T$ arising in Theorem~\ref{GKthmVasilev} must be isomorphic to $S$. 
(Note that the hypotheses of Theorem~\ref{GKthmVasilev} do not guarantee that $H$ is in case (iii) of Theorem~\ref{GKthm}, because $H$ could be a Frobenius group with Frobenius complement $C$ such that $\Gamma(C)$ is not a complete graph.)
\end{remark}

We are able to apply Theorems~\ref{GKthm} and~\ref{GKthmVasilev} to prove Theorem~\ref{thm1} because every sporadic group $G$ has disconnected prime graph. 
Theorems~\ref{GKthm} and \ref{GKthmVasilev} tell us, in particular, that we need to compare $\Gamma(G)$ with $\Gamma(S)$ for various simple groups $S$, in order to determine whether some extension $H$ of $S$ satisfies $\Gamma(H) \cong \Gamma(G)$. 
We therefore also need to know the connected components of the prime graphs of various simple groups. 
As explained in \cite[pp.~187--188]{CamMas}, this data has been catalogued across various papers. 
We summarise what we need in the following theorem, which is based on \cite[Lemma~1.1(c)]{AlekseevaKondratev}. 
(Note that there seem to be some notational inconsistencies in \cite{AlekseevaKondratev}, e.g. in the fourth column of each of \cite[Tables~1--3]{AlekseevaKondratev}, the heading $n_2$, which is defined to be the $2$-part of an integer~$n$, should presumably be $\pi_2(P)$, which is the quantity being listed in that column.) 

\begin{theorem} \label{thm:tables}
If $S$ is a finite simple group with $s(S) \geq 3$, then either
\begin{itemize}
\item $s(S)=3$ and $S$ appears in Table~\ref{tab:s(S)=3}, 
\item $s(S)=4$ and $S$ appears in Table~\ref{tab:s(S)=4}, 
\item $s(S)=5$ and $S = \textnormal{E}_8(q)$ with $q \equiv 0, \pm 1 \pmod 5$, or
\item $s(S)=6$ and $S = \textnormal{J}_4$.
\end{itemize}
\end{theorem}

\begin{table}[!t]
\small
\begin{tabular}{lllll}
\toprule
$S$ & Restrictions on $S$ & $\pi_1(S)$ & $\pi_2(S)$ & $\pi_3(S)$ \\
\midrule
$\text{Alt}_p$ & $p > 6$; $p-2$ prime & $(p-3)!$ & $\{p\}$ & $\{p-2\}$ \\
\midrule
\vspace{4pt} %
$\text{PSL}_2(q)$ & $3 < q \equiv \pm 1 \pmod{4}$ & $q \mp 1$ & $q$ & $\frac{q \pm 1}{2}$ \\
\vspace{4pt} %
$\text{PSL}_2(q)$ & $q>2$, $q$ even & $\{2\}$ & $q-1$ & $q+1$ \\
\vspace{4pt} %
$\text{PSU}_6(2)$ & & $\{2,3,5\}$ & $\{7\}$ & $\{11\}$ \\
\vspace{4pt} %
$\text{P}\Omega_{2p}^-(3)$ & $p = 2^m+1$, $m \geq 1$ & $3(3^{p-1}-1) \Pi_{i=1}^{p-2} (3^{2i}-1)$ & $\frac{3^{p-1}+1}{2}$ & $\frac{3^p+1}{4}$ \\
\vspace{4pt} %
$\text{G}_2(q)$ & $q=3^k$, $k \geq 1$ & $q(q^2-1)$ & $q^2-q+1$ & $q^2+q+1$ \\
\vspace{4pt} %
${}^2\text{G}_2(q)$ & $q = 3^{2m+1}$, $m \geq 1$ & $q(q^2-1)$ & $q-\sqrt{3q}+1$ & $q+\sqrt{3q}+1$ \\
\vspace{4pt} %
$\text{F}_4(q)$ & $q$ even & $q(q^4-1)(q^6-1)$ & $q^4-q^2+1$ & $q^4+1$ \\
${}^2\text{F}_4(q)$ & $q = 2^{2m+1}$, $m \geq 1$ & $q(q^3+1)(q^4-1)$ & $q^2-\sqrt{2q^3}$ & $q^2+\sqrt{2q^3}$ \\
\vspace{4pt} %
& & & $+~q-\sqrt{2q}+1$ & $+~q+\sqrt{2q}+1$ \\
$\text{E}_7(2)$ & & $\{2,3,5,7,11,$ & $\{73\}$ & $\{127\}$ \\
\vspace{2pt} 
& & $13,17,19,31,43\}$ & & \\
$\text{E}_7(3)$ & & $\{2,3,5,7,11,13,$ & $\{757\}$ & $\{1093\}$ \\
& & $19,37,41,61,73,547\}$ & & \\
\midrule
\vspace{2pt} 
$\text{M}_{11}$ & & $\{2,3\}$ & $\{5\}$ & $\{11\}$ \\
\vspace{2pt} 
$\text{HS}$ & & $\{2,3,5\}$ & $\{7\}$ & $\{11\}$ \\
\vspace{2pt} 
$\text{J}_3$ & & $\{2,3,5\}$ & $\{17\}$ & $\{19\}$ \\
\vspace{2pt} 
$\text{Co}_2$ & & $\{2,3,5,7\}$ & $\{11\}$ & $\{23\}$ \\
\vspace{2pt} 
$\text{M}_{23}$ & & $\{2,3,5,7\}$ & $\{11\}$ & $\{23\}$ \\
\vspace{2pt} 
$\text{M}_{24}$ & & $\{2,3,5,7\}$ & $\{11\}$ & $\{23\}$ \\
\vspace{2pt} 
$\text{Suz}$ & & $\{2,3,5,7\}$ & $\{11\}$ & $\{13\}$ \\
\vspace{2pt} 
$\text{Th}$ & & $\{2,3,5,7,13\}$ & $\{19\}$ & $\{31\}$ \\
\vspace{2pt} 
$\text{Fi}_{23}$ & & $\{2,3,5,7,11,13\}$ & $\{17\}$ & $\{23\}$ \\
$\mathbb{B}$ & & $\{2,3,5,7,11,13,17,19,23\}$ & $\{31\}$ & $\{47\}$ \\
\bottomrule
\end{tabular}
\caption{The connected components of $\Gamma(S)$ for the finite simple groups $S$ with $s(S) = 3$. 
In the definitions of $S$, $p$ denotes an odd prime. 
In columns 3--5, the $i$th connected component $\pi_i(S)$ of $\Gamma(S)$ is equal to the given entry if set notation $\{ \cdot \}$ is used, and otherwise $\pi_i(S)$ consists of all of the prime divisors of the given entry.
}
\label{tab:s(S)=3}
\end{table}

\begin{table}[!t]
\small
\begin{tabular}{llllll}
\toprule
$S$ & Restrictions on $S$ & $\pi_1(S)$ & $\pi_2(S)$ & $\pi_3(S)$ & $\pi_4(S)$ \\
\midrule
\vspace{4pt} %
$\text{PSL}_3(4)$ &  & $\{2\}$ & $\{3\}$ & $\{5\}$ & \{7\} \\
\vspace{4pt} %
${}^2\text{B}_2(q)$ & $q=2^{2m+1}$, $m \geq 1$ & $\{2\}$ & $q-1$ & $q-\sqrt{2q}+1$ & $q+\sqrt{2q}+1$ \\
\vspace{4pt} %
${}^2\text{E}_6(2)$ & & $\{2,3,5,7,11\}$ & $\{13\}$ & $\{17\}$ & $\{19\}$ \\
$\text{E}_8(q)$ & $q \equiv 2,3 \pmod{5}$ & $q(q^8-1)(q^{12}-1)$ & $\frac{q^{10}+q^5+1}{q^2+q+1}$ & $q^8-q^4+1$ & $\frac{q^{10}-q^5+1}{q^2-q+1}$ \\
& & $(q^{14}-1)(q^{18}-1)(q^{20}-1)$ & & \\
\midrule
\vspace{2pt} 
$\text{M}_{22}$ & & $\{2,3\}$ & $\{5\}$ & $\{7\}$ & $\{11\}$ \\
\vspace{2pt} 
$\text{J}_1$ & & $\{2,3,5\}$ & $\{7\}$ & $\{11\}$ & $\{19\}$ \\
\vspace{2pt} 
$\text{O'N}$ & & $\{2,3,5,7\}$ & $\{11\}$ & $\{19\}$ & $\{31\}$ \\
\vspace{2pt} 
$\text{Ly}$ & & $\{2,3,5,7,11\}$ & $\{31\}$ & $\{37\}$ & $\{67\}$ \\
\vspace{2pt} 
$\text{Fi}_{24}'$ & & $\{2,3,5,7,11,13\}$ & $\{17\}$ & $\{23\}$ & $\{29\}$ \\
$\mathbb{M}$ & & $\{2,3,5,7,11,13,$ & $\{41\}$ & $\{59\}$ & $\{71\}$ \\
& & $17,19,23,29,31,47\}$ & & & \\
\bottomrule
\end{tabular}
\caption{The connected components of $\Gamma(S)$ for the finite simple groups $S$ with $s(S) = 4$. 
In columns 3--6, $\pi_i(S)$ is equal to the given entry if set notation $\{ \cdot \}$ is used, and otherwise $\pi_i(S)$ consists of all of the prime divisors of the given entry. 
(See Theorem~\ref{thm:tables} for the cases with $s(S) > 4$.)
}
\label{tab:s(S)=4}
\end{table}

Note that Tables~\ref{tab:s(S)=3} and~\ref{tab:s(S)=4} provide only very basic information about $\Gamma(S)$ for each listed simple group $S$, namely the sets of prime divisors comprising the connected components of $\Gamma(S)$. 
In many cases, we need more detailed information about the structure of $\Gamma(S)$. 
This information is collected in several preliminary lemmas given below. 
The proofs of some of these lemmas use the well-known Bang--Zsigmondy lemma, so we state that first. 

\begin{definition} \label{def:ppd}
Let $q$ be a positive integer. 
If $r$ is an odd prime coprime to $q$, then~we define $e(r,q)$ to be the multiplicative order of $q$ modulo $r$, i.e. the least positive integer $n$ such that $r$ divides $q^n-1$. 
If $q$ is odd, then $e(2,q):=1$ if $q \equiv 1 \pmod{4}$, and $e(2,q):=2$ if $q \equiv -1 \pmod{4}$. 
A prime $r$ is said to be a {\em primitive prime divisor} of $q^n-1$ if $e(r,q)=n$. 
\end{definition}

\begin{lemma}[Bang--Zsigmondy \cite{bang,zsigmondy}] \label{lem:zsigmondy}
If $q \geq 2$ and $n \geq 1$ are integers, then there exists a primitive prime divisor of $q^n-1$ unless $(q,n)=(2,6)$, $(2,1)$, or $(3,1)$. 
\end{lemma}

The following lemma concerns the (prime graph of the) sporadic group $\text{J}_4$.

\begin{lemma} \label{lemma:J4}
If $S = \textnormal{J}_4$, then $\Gamma(S)$ is the following graph; in particular, $|\pi(S)|=10$.
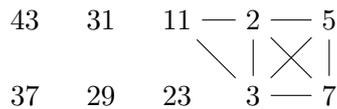
\textnormal{
\begin{center}
\begin{tikzpicture}
\graph [no placement] { 
37[x=-2,y=-0.5]; 43[x=-2,y=0.5]; 31[x=-1,y=0.5]; 
29[x=-1,y=-0.5]; 23[x=0,y=-0.5]; 
11[x=0,y=0.5]; 
3[x=1,y=-0.5]; 2[x=1,y=0.5]; 
7[x=2,y=-0.5]; 5[x=2,y=0.5]; 
2 -- {3,5,7,11}; 
3 -- {5,7,11}; 
5 -- 7
};
\end{tikzpicture}
\end{center}
}
\end{lemma}

\begin{proof}
This is readily verified using the character table of $\text{J}_4$; see e.g. \cite[pp.~188--198]{ATLAS}.
\end{proof}

The following result concerns the simple groups $\text{PSL}_2(q)$, $q \geq 4$. 
It is surely well known, but we were not able to find a proof to cite.

\begin{lemma} \label{lemma:A1}
If $S = \textnormal{PSL}_2(q)$ with $q \geq 4$, then $\Gamma(S)$ consists of three cliques. 
In particular, the maximal size of a coclique in $\Gamma(S)$ is $3$, i.e. $t(S)=3$. 
\end{lemma}

\begin{proof}
Theorem~\ref{thm:tables} tells us that $\Gamma(S)$ has three connected components (see Table~\ref{tab:s(S)=3}). 
It remains to confirm that each one is a clique. 
The maximal subgroups of $S$ are well known; a convenient reference is \cite{giudici}. 
If $q$ is even, then the connected components of $\Gamma(S)$ are $\{2\}$ and the sets of prime divisors of $q \pm 1$. 
Because $S$ contains maximal subgroups $\text{D}_{2(q \pm 1)}$, it contains cyclic subgroups of order $q \pm 1$, so the latter two components are cliques. 
Suppose that $q \neq 3$ is a power of an odd prime $p$. 
If $q \equiv 1 \pmod{4}$, then the connected components of $\Gamma(S)$ are $\{p\}$ and the sets of prime divisors of $\frac{q+1}{2}$ and $q -1$, the latter of which contains $2$. 
If $q=9,$ then each component is an isolated vertex. 
If $q \neq 9$, then $S$ contains maximal subgroups $\text{D}_{q\pm 1}$ and hence cyclic subgroups of order $\frac{q\pm 1}{2}$. 
In particular, the second connected component is a clique. 
The third connected component is also a clique: the odd primes dividing $q -1$ form a clique by the existence of the cyclic subgroup of order $\frac{q-1}{2}$, and they are all adjacent to $2$ because the centraliser of an involution in $G$ is isomorphic to $\text{D}_{q-1}$. 
If $q \equiv -1 \pmod{4}$, then for $q \not \in \{7,11\}$ the same argument applies if we swap the $+$ and $-$ signs. 
If $q \in \{7,11\}$, then each component is an isolated vertex. 
\end{proof}

The following lemma concerns certain orthogonal groups appearing in Table~\ref{tab:s(S)=3}. 

\begin{lemma} \label{lemma:2Dp}
If $S = \textnormal{P}\Omega_{2p}^-(3)$ with $p=2^m+1$ a prime, $m \geq 1$, then $|\pi(S)| \geq p+1$.
\end{lemma}

\begin{proof}
By Theorem~\ref{thm:tables} (Table~\ref{tab:s(S)=3}), $\Gamma(S)$ has three connected components. The component $\pi_1(S)$ consists of all of the prime divisors of $3(3^{p-1}-1) \Pi_{i=1}^{p-2} (3^{2i}-1)$. 
By Lemma~\ref{lem:zsigmondy}, each factor in the product $\Pi_{i=1}^{p-2} (3^{2i}-1)$ has a primitive prime divisor. 
Together with the prime $3$, this gives us at least $p-1$ primes in $\pi_1(S)$. 
(Note that $2$ is a primitive prime divisor of $3^2-1$ by Definition~\ref{def:ppd}.)
Considering also the other two connected components of $\Gamma(S)$, we conclude that there are at least $(p-1)+2=p+1$ primes in $\pi(S)$.
\end{proof}

The remaining preliminary lemmas concern various exceptional groups of Lie type. 
(Note that although we refer to \cite{NMexceptional} for the proofs of Lemmas~\ref{lemma:G2}--\ref{lemma:2F4}, these results were proved in \cite{VasVdo1,VasVdo2}. 
The information that we need is collected in a convenient way in \cite{NMexceptional}.)

\begin{lemma} \label{lemma:E8}
If $S = \textnormal{E}_8(q)$, then $|\pi(S)| \geq 16$, with strict inequality for $q > 2$. 
\end{lemma}

\begin{proof}
See \cite[Lemma~2.9]{CamMas}.
\end{proof}

\begin{lemma} \label{lemma:G2}
If $S = \textnormal{G}_2(3^k)$, $k \geq 1$, then $\pi(S)$ consists of three cliques. 
In particular, the maximal size of a coclique in $\Gamma(S)$ is $3$, i.e. $t(S)=3$. 
\end{lemma}

\begin{proof}
This follows from \cite[Lemma 3.4]{NMexceptional}; see also \cite{VasVdo1,VasVdo2}.
\end{proof}

\begin{lemma} \label{lemma:F4} 
If $S = \textnormal{F}_4(q)$ with $q$ even, then the following statements hold.
\begin{itemize}
\item[(i)] $|\pi(S)| =6$ and $8$ for $q=2$ and $4$, respectively, and $|\pi(S)| \geq 11$ otherwise. 
\item[(ii)] If $q \geq 4$, then $\Gamma(S)$ has at least three vertices of degree at least $5$. 
In particular, $\Gamma(S)$ has at least $12$ edges.
\end{itemize} 
\end{lemma}

\begin{proof}
Assertion~(ii) follows from \cite[Lemma 3.1]{NMexceptional}; see also \cite{VasVdo1,VasVdo2}. 
Assertion~(i) can be verified directly from the order formula $|S| = q^{24}(q^{12}-1)(q^8-1)(q^6-1)(q^2-1)$ for $q = 2^k$ with $k \leq 6$. 
Now suppose that $k \geq 7$. 
By Table~\ref{tab:s(S)=3}, the connected component $\pi_1(S)$ of $\Gamma(S)$ consists of all of the prime divisors of $q(q^4-1)(q^6-1)$. 
Note that $\phi_i := 2^i-1$ divides $q(q^4-1)(q^6-1)$ for all $i \in \{2,3,4,k,2k,4k,3k,6k\}$. 
Because $k \geq 7$, the eight $\phi_i$ are distinct, and Lemma~\ref{lem:zsigmondy} implies that each of them has a primitive prime divisor. 
Together with the prime $2$, this gives us at least nine primes in $\pi_1(S)$. 
Considering also the other two connected components of $\Gamma(S)$ yields $|\pi(S)| \geq 11$. 
\end{proof}

\begin{lemma} \label{lemma:2F4}
If $S ={}^2\textnormal{F}_4(q)$ with $q = 2^{2m+1}$, $m \geq 1$, then $|\pi(S)| \geq 8$ and every coclique in $\Gamma(S)$ has size at most $5$, i.e. $t(S) \leq 5$. 
Moreover, if $m \geq 2$, then $\Gamma(S)$ contains the following graph as a subgraph for some $p_1,\dots,p_6 \in \pi(S) \setminus \{2,3\}$.
\textnormal{
\begin{center}
\begin{tikzpicture}
\graph [no placement, math nodes] { 
2[x=1,y=-0.5]; 
3[x=2,y=0.5]; 
p_1[x=-1,y=-0.5]; 
p_2[x=-1,y=0.5]; 
p_4[x=1,y=0.5] -- {2,3};  
p_3[x=0,y=0.5] -- {p_4,2};  
p_5[x=2,y=-0.5] -- {2,3};
p_6[x=3,y=0.5] -- 3; 
p_4 -- p_5;
2 -- 3
};
\end{tikzpicture}
\end{center}
}
\end{lemma}

\begin{proof}
This follows from \cite[Lemma 3.3]{NMexceptional}, but we clarify some details; see also \cite{VasVdo1,VasVdo2}. 
The graph pictured above is the so-called ``compact form'' of $\Gamma(S)$. 
This means that each vertex other than $2$ and $3$ in the above graph represents a certain {\em set} of prime divisors of $|S|$ that happen to form a clique in $\Gamma(S)$, and an edge indicates that {\em all} edges from one clique to another are present in $\Gamma(S)$. 
In particular, every coclique in $\Gamma(S)$ has size at most $5$. 
Specifically, $p_1,\dots,p_6$ represent the following sets of prime divisors of $|S|$:
\begin{itemize}
\item $p_1$ and $p_2$ represent the sets of prime divisors of $q^2 \pm \sqrt{2q^3} + q \pm \sqrt{2q} + 1$; 
\item $p_3$ represents the set of prime divisors of $q^2+1$; 
\item $p_4$ represents the set of prime divisors of $q-1$; 
\item $p_5$ represents the set of prime divisors of $q+1$ other than $3$; and 
\item $p_6$ represents the set of prime divisors of $q^2-q+1$ other than $3$. 
\end{itemize}
Note that these sets are disjoint, and that the sets represented by $p_1$, $p_2$, $p_3$, and $p_4$ are non-empty for $q = 2^{2m+1}$, $m \geq 1$. 
The set represented by $p_6$ is also non-empty. 
Indeed, $q^6-1 = (q^3-1)(q-1)(q^2-q+1)$, so $q^2-q+1$ is divisible by a primitive prime divisor of $q^6-1 = 2^{6(m+1)}-1$, which exists by Lemma~\ref{lem:zsigmondy} and is not equal to $3$ because $3$ is a primitive prime divisor of $2^2-1$. 
The set represented by $p_5$ is non-empty if and only if $m \geq 2$. 
Indeed, this set is empty if and only if $2^{2m+1}+1 = 3^k$ for some $k$, and the main result of \cite{gerono} shows that this occurs if and only if $k=2$ and $m=1$. 
Therefore, if $m \geq 2$, then every `vertex' in the compact form of $\Gamma(S)$ is non-empty, so $\Gamma(S)$ contains a subgraph of the indicated form; in particular, $|\pi(S)| \geq 8$. 
If $m=1$, then $\Gamma(S)$ is the graph
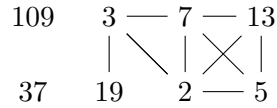
\begin{center}
\begin{tikzpicture}
\graph [no placement] { 
2[x=1,y=-0.5]; 
13[x=2,y=0.5]; 
37[x=-1,y=-0.5]; 
109[x=-1,y=0.5]; 
19[x=0,y=-0.5]; 
7[x=1,y=0.5] -- {2,13};  
3[x=0,y=0.5] -- {7,2};  
5[x=2,y=-0.5] -- {2,13};
7 -- 5;
2 -- 3; 
2 --13; 
3 -- 19
};
\end{tikzpicture}
\end{center}
which has $8$ vertices and no coclique of size larger than $4$. 
\end{proof}

The following two lemmas concern the Suzuki groups ${}^2\text{B}_2(2^{2m+1})$, $m \geq 1$.

\begin{lemma} \label{lemma:suz_primes}
If $q = 2^{2m+1}$, $m \geq 1$, then both of $Q_\pm := q \pm \sqrt{2q} + 1$ are prime powers if and only if $m \in \{1,2\}$.
\end{lemma}

\begin{proof}
Note that $Q_\pm$ are coprime. 
Because $Q_+Q_- = q^2+1 \equiv 0 \pmod 5$, one of $Q_\pm$ is divisible by $5$. 
If $Q_\pm$ are both prime powers, then one of $Q_\pm = 5^\ell$ holds for some~$\ell$. 
Therefore, one of $5^\ell -1 = 2^{m+1}(2^m \pm 1)$ holds. 
Write $\ell = 2^kr$, where $k \geq 0$ and $r$ is odd. 
By the well-known ``lifting the exponent lemma'', see e.g. \cite[Lemma~A.4]{BGbook}, we then have $5^\ell-1 = 2^{k+2}r'$ for some odd $r'$. 
Hence, $\log_2(5^\ell) = 2^{m-1}r\log_2(5) > 2^m$. 
On the other hand, $Q_- < Q_+ < 2^{2m+2}$ for all $m \geq 1$, so $\log_2(5^\ell) = \log_2(Q_\pm) < 2m+2$. 
In particular, $2^m < 2m+2$, so $m \leq 2$. 
Conversely, $Q_\pm$ are both prime powers if $m \in \{1,2\}$. 
\end{proof}

\begin{lemma} \label{lemma:suz}
If $S = {}^2\textnormal{B}_2(2^{2m+1})$, $m \geq 1$, then $\Gamma(S)$ consists of four cliques (one of which is the isolated vertex $2$). 
In particular, the maximal size of a coclique in $\Gamma(S)$ is $4$, i.e. $t(S)=4$.  
Moreover, $\Gamma(S)$ consists of four isolated vertices if and only if $m \in \{1,2\}$. 
\end{lemma}

\begin{proof}
Note that Theorem~\ref{thm:tables} already asserts that $\Gamma(S)$ has four connected components; see Table~\ref{tab:s(S)=4}.
The order of $S$ is $q^2(q-1)(q+\sqrt{2q}+1)(q-\sqrt{2q}+1)$, where $q=2^{2m+1}$. 
These four factors of $|S|$, call them $k_i$ for $i \in \{1,2,3,4\}$, are pairwise coprime, so prime divisors of distinct $k_i$ are non-adjacent in $\Gamma(S)$. 
On the other hand, by \cite[Table~8.16]{BHR-D}, $S$ contains a cyclic subgroup (indeed, a maximal torus) of order $k_i$ for each $i \in \{2,3,4\}$, so prime divisors of a given $k_i$ are adjacent. 
The final assertion follows from Lemma~\ref{lemma:suz_primes}. 
\end{proof}

The following lemma is well known, but we include a proof for convenience; it is used to prove the subsequent lemma about the Suzuki groups ${}^2\text{B}_2(2^{2m+1})$.

\begin{lemma} \label{lemma:brauer}
Let $G$ be a finite group, $F$ a field of characteristic $p>0$, and $V$ an irreducible $FG$-module with corresponding Brauer character $\chi$. 
If $g\in G$ has prime order not equal to $p$, then the fixed-point space of $g$ on $V$ has dimension $\frac{1}{|g|} \sum _{x \in \langle g \rangle} \chi(x)$. 
\end{lemma}

\begin{proof}
Consider the restriction of $V$ to $H=\langle g \rangle$, with its corresponding Brauer character $\chi_H$. 
Because $|H|$ is coprime to $p$, the dimension of the fixed-point space of $g$ is precisely the multiplicity of the trivial character $1_H$ as a constituent of $\chi_H$. 
This is equal to the inner product of $\chi_H$ and $1_H$, namely  
$[\chi_H, 1_H] = \frac{1}{|H|} \sum_{x\in H}  \chi_H(x) 1_H(x) = \frac{1}{|g|} \sum_{x\in \langle g \rangle}  \chi(x)$.
\end{proof}

\begin{lemma} \label{lemma:suz_reps}
If $S={}^2\textnormal{B}_2(2^{2m+1})$, $m\geq 1$, and $V$ is an irreducible module for $S$ over a finite field $\mathbb{F}_r$ with $r$ coprime to $|S|$, then for every prime $p$ dividing $|S|$ there exists an element of order $p$ in $S$ that fixes a non-zero vector in $V$.
\end{lemma}

\begin{proof}
We apply Lemma~\ref{lemma:brauer} with $G=S$ and $F = \mathbb{F}_r$. 
Because $r$ is coprime to $|S|$, the Brauer character table of $S$ in characteristic $r$ is identical to the ordinary character table of $S$. 
The ordinary character table of $S$ was computed by Suzuki \cite{Suzuki}, and is conveniently reproduced by Martineau~\cite{MartineauOdd}. 
A routine calculation establishes the result. 
\end{proof}

The final two preliminary lemmas concern the small Ree groups ${}^2\text{G}_2(3^{2m+1})$, $m \geq 1$. 

\begin{lemma}\label{lemma:2G2}
Let $q=3^{2m+1}$, $m\geq 1$. 
If $Q_\pm := q\pm \sqrt{3q}+1$ is a prime power, then it is a prime. 
Moreover, if both of $Q_\pm$ are prime powers (hence primes), then $m \equiv 1 \pmod{3}$.
\end{lemma}
\begin{proof}
The proof is adapted from \cite{StackExchange}. 
Suppose that one of $Q_{\pm}  = p^k$ for some prime $p$ and some $k \geq 1$. 
Note that $p \neq 3$, and that
\[
3^{m+1}(3^m\pm 1) = p^k-1, \quad \text{so in particular} \quad 3^{m+1} > p^{k/2}.
\]
If $k$ is even, then $3^{m+1}$ divides one of $p^{k/2} \pm 1$. 
This means that either $p^k-1 > 3^{2m+2}$, contradicting the above inequality, or $3^{m+1} = p^{k/2}+1$, implying that $3^m\pm 1 = 3^{m+1}-2$, which is impossible. 
Therefore, $k$ is odd.
If $p \equiv 2 \pmod 3$, then $p^k-1\equiv 1 \pmod 3$, a contradiction. 
Hence, $p-1 = 2s\cdot 3^t$ for some $t\geq 1$ and some $s$ coprime to 3. 
The ``lifting the exponent lemma" \cite[Lemma~A.4]{BGbook} therefore implies that $k  = 3^ab$, where $a = m+1-t$ and $b$ is coprime to 3.  Because $3^{m+1}> p^{k/2}$ (as above), we have
\begin{equation} \label{2G2lemmaEq2}
a+t = m+1 > \log_3(p^{k/2}) = \frac{3^ab}{2}\log_3(2s\cdot 3^t) > \frac{3^abt}{2}.
\end{equation}
Because $t \geq 1$, this implies that $\frac{a}{t} > \frac{3^ab}{2}-1$, and hence that $a > \frac{3^a}{2}-1$. 
Therefore, $a \leq 1$. 
If $a=1$, then $1+t > \frac{3^bt}{2}$, so $b=t=1$ because $b \neq 0$.
The inequality $a+t > \frac{3^ab}{2}\log_3(2s\cdot 3^t)$ from \eqref{2G2lemmaEq2} then implies that $\frac{4}{3} > \log_3(6s) \geq \log_3(6) \approx 1.63$, a contradiction. 
Therefore, $a=0$, so \eqref{2G2lemmaEq2} implies that $t > \frac{bt}{2}$. 
Hence, $b=1$ and $k = 3^ab = 1$, so $Q_\pm = p^k$ is a prime. 

It remains to prove the final assertion. 
Recall that $q = 3^{2m+1}$. 
One can check that $Q_+Q_- = q^2-q+1$ is divisible by $7$ if and only if $m \not \equiv 1 \pmod 3$. 
Therefore, if $Q_\pm$ are both prime powers (hence primes), then either $m \equiv 1 \pmod 3$, or $m \not \equiv 1 \pmod 3$ and $Q_- = 7$. 
The latter case implies that $m=0$, contradicting our assumption that $m \geq 1$, so the former case holds as asserted. 
\end{proof}

\begin{lemma} \label{lemma:2G2_2}
If $S = {}^2\textnormal{G}_2(3^{2m+1})$, $m \geq 1$, then every coclique in $\Gamma(S)$ has size at most~$5$, i.e. $t(S) \leq 5$. 
Moreover, if $m \geq 2$ and $\Gamma(S)$ contains two isolated vertices, then $|\pi(S)| \geq 9$.
\end{lemma}

\begin{proof}
The fact that every coclique in $\Gamma(S)$ has size at most $5$ can be seen from the compact form of $\Gamma(S)$ (as defined in the proof of Lemma~\ref{lemma:2F4}), which, per \cite[Fig.~1]{Zavar}, is isomorphic to the graph shown below; see also \cite[Table~2]{VasVdo2}.
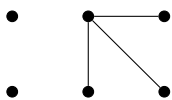
\begin{center}
\begin{tikzpicture}
\graph [no placement, empty nodes, nodes={circle,draw,fill=black,inner sep=0pt,minimum size=4pt}] { 
C[x=-1,y=-0.5]; 
D[x=-1,y=0.5];
2[x=1,y=-0.5]; 
3[x=1,y=0.5];  
A[x=0,y=-0.5];  
B[x=0,y=0.5] -- {2,3,A}
};
\end{tikzpicture}
\end{center}
It remains to prove the second assertion. 
Suppose that $m \geq 2$ and that $\Gamma(S)$ contains two isolated vertices. 
The latter assumption implies that the connected components of $\Gamma(S)$ consisting of the prime divisors of $q \pm \sqrt{3q} + 1$ must be prime powers. 
(This is immediate from Table~\ref{tab:s(S)=3}, given that the remaining connected component $\pi_1(S)$ clearly contains at least the vertices $2$ and $3$.) 
Lemma~\ref{lemma:2G2} therefore implies that $m \equiv 1 \pmod{3}$. 
Hence, $m = 3k+1$, where $k \geq 1$ because we are assuming that $m \geq 2$. 
Per Table~\ref{tab:s(S)=3}, the connected component $\pi_1(S)$ consists of $2$, $3$, and the odd prime divisors of $q^2-1$. 
We need to show that $q^2-1$ is divisible by at least five distinct odd primes. 
Given that $m = 3k+1$ with $k \geq 1$, we have $2m+1 = 3\ell$, where $\ell = 2k+1 \geq 3$. 
Therefore, $q^2 - 1 = 3^{6\ell} - 1$, which is divisible by $\phi_i = 3^i-1$ for all $i \in \{ 3, \ell, 2\ell, 3\ell, 6\ell \}$. 
Each of these $\phi_i$ has a primitive prime divisor by Lemma~\ref{lem:zsigmondy}, and none of these primitive prime divisors is equal to $2$ because $2$ is a primitive prime divisor of $3^2-1$. 
Therefore, $q^2-1$ is certainly divisible by at least five distinct primes, except possibly when $\ell=3$. 
For $\ell=3$, namely $q = 3^9$, a direct calculation shows that $q^2-1$ has exactly five distinct odd prime divisors. 
\end{proof}

\begin{table}[!t]
\small
\begin{tabular}{ll}
\toprule
Recognisability & Group \\
\midrule
Recognisable & $\mathbb{B}$, $\text{Co}_1$, $\text{Co}_2$, $\text{Fi}_{23}$, $\text{Fi}_{24}'$, $\text{J}_1$, $\text{J}_3$, $\text{J}_4$, $\text{Ly}$, $\text{M}_{22}$, $\text{M}_{23}$, $\text{M}_{24}$, $\mathbb{M}$, $\text{O'N}$, $\text{Suz}$, $\text{Th}$ \\
$2$-recognisable & $\text{HN}$, $\text{HS}$, $\text{M}_{11}$ \\
$3$-recognisable & $\text{Fi}_{22}$ \\
Unrecognisable & $\text{Co}_3$, $\text{He}$, $\text{J}_2$, $\text{M}_{12}$, $\text{McL}$, $\text{Ru}$ \\
\bottomrule
\end{tabular}
\caption{Recognisability of the sporadic simple groups by their labelled prime graphs. 
For proofs, see \cite{Hagie, Kon1,Kon2,MBCo1,MazurovShi,PraegerShi,Zavar}. 
Note that the group $\text{Ru}$ is unrecognisable, contrary to the assertion made in \cite{Kon1} that it is recognisable; cf. Remark~\ref{rem:Ru}. 
This is proved in Section~\ref{ss:Ru}.}
\label{tabSummary}
\end{table}

\section{Proof of Theorem~\ref{thm1} --- The unrecognisable groups} \label{s:unrec}

Here we prove the ``only if'' implication of the first assertion of Theorem~\ref{thm1}, and the second assertion of the theorem. 
That is, we show that if $G$ is a sporadic group {\em not} listed in Theorem~\ref{thm1}, then $G$ is unrecognisable by the isomorphism type of its prime graph, i.e. there are infinitely many pairwise non-isomorphic groups $H$ with $\Gamma(H) \cong \Gamma(G)$. 

By Table~\ref{tabSummary}, the groups 
\[
\textnormal{Co}_3,\; \textnormal{He},\; \textnormal{J}_2,\; \textnormal{M}_{12},\; \textnormal{McL}
\] 
are already known to be unrecognisable by their labelled prime graphs, so they are certainly unrecognisable by the isomorphism types of their prime graphs. 
Per Remark~\ref{rem:Ru}, the group $\text{Ru}$ is also unrecognisable by its labelled prime graph, but this requires a proof, which we provide in Section~\ref{ss:Ru}. 
In addition, it remains to show that the following groups are unrecognisable by the isomorphism types of their prime graphs:
\[
\textnormal{Co}_1,\; \textnormal{Co}_2,\; \textnormal{Fi}_{22},\; \textnormal{HN},\; \textnormal{HS},\; \textnormal{J}_1,\; \textnormal{J}_3,\; \textnormal{M}_{11},\; \textnormal{M}_{22},\; \textnormal{M}_{23},\; \textnormal{M}_{24},\; \textnormal{Suz}.
\] 

By \cite[Theorem~1.2]{CamMas}, a group $G$ is unrecognisable by its labelled prime graph if and only if there is a group $H$ with non-trivial soluble radical such that $\Gamma(H) = \Gamma(G)$. 
It therefore suffices to exhibit, for each sporadic group $G$ in the second list above, one group $H$ that has non-trivial soluble radical and satisfies $\Gamma(H) \cong \Gamma(G)$. 
Alternatively, we can simply exhibit a group $H$ with trivial soluble radical that satisfies $\Gamma(H) \cong \Gamma(G)$ and is known to be unrecognisable by its labelled prime graph. 

We need to be able to draw or otherwise describe the prime graphs of the above groups and various other groups. 
The prime graph of a group $G$ can be determined from the character table of $G$. 
The character tables of the sporadic groups and many other groups that arise in our proofs are available in the Atlas~\cite{ATLAS}, \textsf{GAP}~\cite{GAPbc,GAP} and/or \textsc{Magma}~\cite{Magma}.

\subsection{The cases $G = \textnormal{J}_1$ and $G = \textnormal{M}_{22}$} \label{ss:J1M22}
We show that in both of these cases, $\Gamma(G)$ is isomorphic to the prime graph of some Suzuki group ${}^2\text{B}_2(2^{2m+1})$. 
This is sufficient because the Suzuki groups are unrecognisable by their labelled prime graphs \cite[Proposition~4.1]{NMexceptional}. 
If $G = \textnormal{J}_1$, then $\Gamma(G)$ consists of the clique $\{2,3,5\}$ and the isolated vertices $7$, $11$, and $19$. 
By Table~\ref{tab:s(S)=4} and Lemma~\ref{lemma:suz}, the prime graph of the Suzuki group ${}^2\text{B}_2(2^{19})$ consists of the clique $\{5,229,457\}$ and the isolated vertices $2$, $524287$, and $525313$. 
Therefore, $\Gamma(G) \cong \Gamma({}^2\text{B}_2(2^{19}))$. 
If $G = \textnormal{M}_{22}$, then $\Gamma(G)$ consists of the edge $\{2,3\}$ and the isolated vertices $5$, $7$, and $11$. 
The prime graph of ${}^2\text{B}_2(2^7)$ consists of the edge $\{5,29\}$ and the isolated vertices $2$, $113$, and $127$. 
Therefore, $\Gamma(G) \cong \Gamma({}^2\text{B}_2(2^7))$.

\subsection{The case $G = \textnormal{M}_{11}$} \label{ss:M11}
We show that the prime graph of $\textnormal{M}_{11}$ is isomorphic to the prime graph of $\text{G}_2(3)$. 
This is sufficient because $\text{G}_2(3)$ is unrecognisable by its labelled prime graph \cite[Proposition~5.2]{NMexceptional}. 
The prime graphs of $\text{M}_{11}$ and $\text{G}_2(3)$ each consist of the edge $\{2,3\}$, the isolated vertex $7$, and one other isolated vertex, namely $5$ or $13$, respectively.

\subsection{The cases $G = \textnormal{HS}$ and $G = \textnormal{J}_{3}$} \label{ss:HS-J3}
If $G = \textnormal{HS}$ or $G = \textnormal{J}_{3}$, then $\Gamma(G)$ consists of the clique $\{2,3,5\}$ and two isolated vertices: $7$ and $11$ in the former case, and $17$ and $19$ in the latter. 
We claim that, in both cases, $\Gamma(G)$ is isomorphic to the prime graph of a certain group of shape $\mathbb{F}_4^{30} {:} \text{PSL}_2(61)$. 
The prime graph of $\text{PSL}_2(61)$ consists of the clique $\{2,3,5\}$ and the isolated vertices $31$ and $61$. 
The $2$-modular character table of $\text{PSL}_2(61)$, which can be computed in \textsc{Magma}, shows that $\text{PSL}_2(61)$ admits a $30$-dimensional module over $\mathbb{F}_4$ on which every element of order $31$ or $61$ fixes no non-zero vector. 
This can be verified by applying Lemma~\ref{lemma:brauer}, as can several similar claims made throughout the paper. 
(We shall not explicitly cite Lemma~\ref{lemma:brauer} in each case.) 
In other words, in the corresponding affine group $H = \mathbb{F}_4^{30} {:} \text{PSL}_2(61)$ (i.e. the semidirect product of the aforementioned module $\mathbb{F}_4^{30}$, regarded as an abelian group, by $ \text{PSL}_2(61)$), no element of order $31$ or $61$ centralises an element of order~$2$. 
Therefore, $\Gamma(H) = \Gamma(\text{PSL}_2(61)) \cong \Gamma(G)$. 
Given that $H$ has non-trivial soluble radical, it follows that that $G$ is unrecognisable by the isomorphism type of $\Gamma(G)$.

\subsection{The case $G = \textnormal{M}_{23}$} \label{ss:M23}
The prime graph of $G = \text{M}_{23}$ is shown below.
\begin{center}
\begin{tikzpicture}
\graph [no placement] { 23[x=0,y=1]; 11[x=0,y=0] ; 2[x=1,y=1] -- 3[x=2,y,=1]; 3 -- 5[x=2,y=0]; 7[x=1,y=0] -- 2 };
\end{tikzpicture}
\end{center}
Let $S = {}^2\text{B}_2(q)$ with $q=2^7$. 
We claim that $\Gamma(G)$ is isomorphic to the prime graph of a certain affine group with point stabiliser $\text{Aut}(S)$. 
As noted in Section~\ref{ss:J1M22}, $\Gamma(S)$ consists of the edge $\{5,29\}$ and the isolated vertices $2$, $113$, and $127$. 
The outer automorphism group of $S$, which consists entirely of field automorphisms, is cyclic of order $7$. 
If $\sigma$ is a generator of $\text{Out}(S)$, then $C_{\text{Aut}(S)}(\sigma) \cap S \cong {}^2\text{B}_2(2) \cong 5{:}4$, so $\Gamma(\text{Aut}(S))$ is obtained from $\Gamma(S)$ by adding the edges $\{2,7\}$ and $\{5,7\}$. 
In particular, $\Gamma(\text{Aut}(S)) \cong \Gamma(G)$. 
Let $W$ be the natural module for $\text{Sp}_4(q)$. 
When viewed as a subgroup of $\text{Sp}_4(q)$, the group $S$ acts irreducibly on $W$. 
The group $\text{Aut}(S)$ acts irreducibly on $V = \oplus_{i=0}^6 W^{(i)}$, where $W^{(i)}$ is the module obtained by composing the representation of $S$ corresponding to $W$ with the $i$th power of the automorphism $x \mapsto x^2$ of $\mathbb{F}_q$. 
The modules $W^{(i)}$ can be constructed in \textsc{Magma} using the function \texttt{SuzukiIrreducibleRepresentation}, whence one can check that no element of odd prime order in $\text{Aut}(S)$ fixes a non-zero vector in $V$.
Therefore, the affine group $H = V {:} \text{Aut}(S)$ has the same prime graph as $\text{Aut}(S)$. 
In particular, $\Gamma(G) \cong \Gamma(H)$.

\subsection{The case $G = \textnormal{Co}_2$} \label{ss:Co2}
The prime graph of $G=\textnormal{Co}_2$ is shown below.
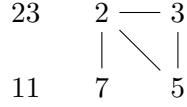
\begin{center}
\begin{tikzpicture}
\graph [no placement] { 23[x=0,y=1]; 11[x=0,y=0]; 2[x=1,y=1];  {3[x=2,y,=1],5[x=2,y=0]} -- 2; 7[x=1,y=0] -- 2 ; 3 -- 5 };
\end{tikzpicture}
\end{center}
Let $S = {}^2\text{B}_2(q)$ with $q=2^{19}$. 
We claim that $\Gamma(G)$ is isomorphic to the prime graph of a certain affine group with point stabiliser $S$. 
Recall from Section~\ref{ss:J1M22} that $\Gamma(S)$ consists of the clique $\{5,229,457\}$ and the isolated vertices $2$, $524287$, and $525313$. 
Let $W$ be the natural module for $\text{Sp}_4(q)$. 
When viewed as a subgroup of $\text{Sp}_4(q)$, the group $S$ acts irreducibly on $W$. 
It also acts irreducibly on the module $W^{(1)}$ obtained by composing the representation of $S$ corresponding to $W$ with the field automorphism $x \mapsto x^2$. 
By \cite[Lemma~1]{MartineauEven}, $V := W \otimes W^{(1)}$ is also an irreducible $\mathbb{F}_qS$-module. 
The module $V$ can be constructed in \textsc{Magma} as described in Section~\ref{ss:M23}, and one can check that the only elements of odd prime order $S$ that fix a non-zero vector in $V$ are elements of order $5$. 
Therefore, the prime graph of the affine group $H = V {:} S$ is obtained from $\Gamma(S)$ by adding the edge $\{2,5\}$, which yields a graph isomorphic to $\Gamma(G)$.

\subsection{The cases $G = \textnormal{M}_{24}$ and $G = \textnormal{Suz}$} \label{ss:M24-Suz}
If $G = \textnormal{M}_{24}$ or $G = \textnormal{Suz}$, then $\Gamma(G)$ is the graph
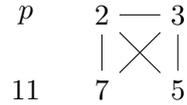
\begin{center}
\begin{tikzpicture}
\graph [no placement] { $p$[x=0,y=1]; 11[x=0,y=0]; 2[x=1,y=1];  {3[x=2,y,=1],5[x=2,y=0]} -- 2; 7[x=1,y=0] -- {2,3} ; 3 -- 5 };
\end{tikzpicture}
\end{center}
with $p=13$ or $23$ respectively. 
Consider again the group $S = {}^2\text{B}_2(q)$ with $q=2^{19}$. 
In the notation of Section~\ref{ss:Co2}, \cite[Lemma~1]{MartineauEven} implies that $V := W \otimes W^{(1)} \otimes W^{(4)}$ is an irreducible $\mathbb{F}_qS$-module (where $W^{(4)}$ is obtained by composing the representation of $S$ corresponding to $W$ with the fourth power of the Frobenius automorphism). 
The module $V$ can be constructed in \textsc{Magma}, and one can check that the only elements of odd prime order in $S$ that fix a non-zero vector in $V$ are elements of order $5$ or $229$. 
Therefore, the prime graph of the affine group $H = V {:} S$ is obtained from $\Gamma(S)$ by adding the edges $\{2,5\}$ and $\{2,229\}$, which yields a graph isomorphic to $\Gamma(G)$.

\subsection{The case $G = \textnormal{Ru}$} \label{ss:Ru}
As noted in Remark~\ref{rem:Ru}, we need to show that $G = \textnormal{Ru}$ is unrecognisable by its labelled prime graph, which is as follows.
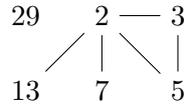
\begin{center}
\begin{tikzpicture}
\graph [no placement] { 29[x=0,y=1]; 13[x=0,y=0] -- 2[x=1,y=1];  {3[x=2,y,=1],5[x=2,y=0]} -- 2; 7[x=1,y=0] -- 2 ; 3 -- 5 };
\end{tikzpicture}
\end{center}
The $2$-modular character table of $G$, which is available in \textsf{GAP}, shows (via Lemma~\ref{lemma:brauer}) that $G$ admits a $28$-dimensional module over $\mathbb{F}_2$ such that every element of order $29$ acts fixed-point freely on non-zero vectors. 
The corresponding affine group $H=\mathbb{F}_2^{28}{:}G$ therefore satisfies $\Gamma(H) = \Gamma(G)$, so \cite[Theorem~1.2]{CamMas} implies that $G$ is unrecognisable by its labelled prime graph. 
(The error in \cite{Kon1} arose because the author incorrectly asserted that $G$ does {\em not} admit such a module.)

\subsection{The case $G = \textnormal{Fi}_{22}$} \label{ss:Fi22}
The prime graph of $G = \textnormal{Fi}_{22}$ is shown below.
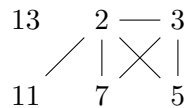
\begin{center}
\begin{tikzpicture}
\graph [no placement] { 13[x=0,y=1]; 11[x=0,y=0] -- 2[x=1,y=1];  {3[x=2,y,=1],5[x=2,y=0]} -- 2; 7[x=1,y=0] -- {2,3} ; 3 -- 5 };
\end{tikzpicture}
\end{center}
We claim that $\Gamma(G)$ is isomorphic to the prime graph of a certain group of shape $\mathbb{F}_3^{22}{:}\text{M}_{24}$. 
The prime graph of $\text{M}_{24}$ is shown in Section~\ref{ss:M24-Suz}, with $p=23$. 
According to the $3$-modular character table of $\text{M}_{24}$ in \textsf{GAP}, there is a $22$-dimensional module for $\text{M}_{24}$ over $\mathbb{F}_3$ on which elements of order $11$ fix some non-zero vector and elements of order $23$ fix no non-zero vector. 
If $H = \mathbb{F}_3^{22}{:}\text{M}_{24}$ is the corresponding affine group, then $\Gamma(H)$ is obtained from $\Gamma(\text{M}_{24})$ by adding the edge $\{3,11\}$, so $\Gamma(H) \cong \Gamma(G)$.

\subsection{The case $G = \textnormal{HN}$} \label{ss:HN}
The prime graph of $G = \textnormal{HN}$ is shown below.
\begin{center}
\begin{tikzpicture}
\graph [no placement] { 19[x=0,y=1]; 11[x=0,y=0] -- 2[x=1,y=1];  {3[x=2,y,=1],5[x=2,y=0]} -- 2; 7[x=1,y=0] -- {2,3,5} ; 3 -- 5 };
\end{tikzpicture}
\end{center}
We claim that $\Gamma(G)$ is isomorphic to the prime graph of a certain group of shape $\mathbb{F}_3^{12}{:}\text{Alt}_{13}$. 
The prime graph of $\text{Alt}_{13}$ consists of the clique $\{2,3,5,7\}$ and the isolated vertices $11$ and $13$.
According to the $3$-modular character table of $\text{Alt}_{13}$ in \textsf{GAP}, there is a $12$-dimensional module for $\text{Alt}_{13}$ over $\mathbb{F}_3$ on which elements of order $11$ fix some non-zero vector and elements of order $13$ fix no non-zero vector. 
If $H = \mathbb{F}_3^{12}{:}\text{Alt}_{13}$ is the corresponding affine group, then $\Gamma(H)$ is obtained from $\Gamma(\text{Alt}_{13})$ by adding the edge $\{3,11\}$, so $\Gamma(H) \cong \Gamma(G)$.

\subsection{The case $G = \textnormal{Co}_1$} 
The prime graph of $G = \textnormal{Co}_1$ is shown below.
\begin{center}
\begin{tikzpicture}
\graph [no placement] { 
23[x=-1,y=0]; 
13[x=0,y=-0.5]; 11[x=0,y=0.5]; 
3[x=1,y=-0.5]; 2[x=1,y=0.5]; 
7[x=2,y=-0.5]; 5[x=2,y=0.5]; 
2 -- {3,5,7,11,13}; 
3 -- {5,7,11,13}; 
5 -- 7
};
\end{tikzpicture}
\end{center}
Consider the group $H = (\mathbb{F}_2^{11} \times \mathbb{F}_r^{22}){:}\text{M}_{23}$, where $\mathbb{F}_2^{11}$ is either of the modules for $\text{M}_{23}$ described in Section~\ref{ss:Ru}, and $\mathbb{F}_r^{22}$ is a faithful irreducible module of dimension $22$ for $\text{M}_{23}$ over some field $\mathbb{F}_r $ of order $r$ coprime to $|\text{M}_{23}|$. 
We claim that $\Gamma(G) \cong \Gamma(H)$. 
By the argument in Section~\ref{ss:Ru}, the subgroup $\mathbb{F}_2^{11}{:}\text{M}_{23}$ of $H$ yields the subgraph
\begin{center}
\begin{tikzpicture}
\graph [no placement] { 
23[x=-1,y=0]; 
11[x=0,y=0.5]; 7[x=0,y=-0.5]; 
$r$[x=1,y=-0.5]; 5[x=2,y=0.5]; 
2[x=1,y=0.5]; 3[x=2,y=-0.5]; 
2 -- {3,5,7,11}; 
3 -- 5
};
\end{tikzpicture}
\end{center}
of $\Gamma(H)$. 
(The vertex $r$ is not part of this subgraph, but we include it in the picture for clarity.)
To obtain $\Gamma(H)$ itself, it remains to add the edge $\{2,r\}$ and an edge $\{p,r\}$ for every odd $p \in \pi(\text{M}_{23})$ such that some element of order $p$ in $\text{M}_{23}$ fixes a non-zero vector in the module $\mathbb{F}_r^{22}$. 
The character table of $\text{M}_{23}$ shows that elements of all prime orders other than $23$ fix a non-zero vector in $\mathbb{F}_r^{22}$, so we end up with a graph isomorphic to $\Gamma(G)$.

\section{Proof of Theorem~\ref{thm1} --- The recognisable groups} \label{s:rec}

We now complete the proof of Theorem~\ref{thm1} by showing that if $G$ is any of the groups
\begin{equation} \label{recogG}
\mathbb{B},\; \textnormal{Fi}_{23},\; \textnormal{Fi}_{24}',\; \textnormal{J}_4,\; \textnormal{Ly},\; \mathbb{M},\; \textnormal{O'N},\; \textnormal{Th},
\end{equation}
then $G$ is recognisable by the isomorphism type of its prime graph.

The notation established in Section~\ref{s:prelim} is used freely. 
The argument for each group begins with the following observation based on Theorems~\ref{GKthm} and~\ref{GKthmVasilev}.

\begin{lemma} \label{lemma:obs}
If $G$ is one of the groups appearing in \eqref{recogG} and $H$ is a group not isomorphic to $G$ such that $\Gamma(H) \cong \Gamma(G)$, then all of the following statements hold. 
\begin{itemize}
\item[(i)] There exists a non-abelian simple group $S$ not isomorphic to $G$ such that 
\[
S \trianglelefteq \overline{H} := H/F(H) \leq \textnormal{Aut}(S).
\] 
Moreover, $s(S) \geq s(G)$, and $\Gamma(S)$ has at most as many vertices and at most as many edges as $\Gamma(G)$. 
\item[(ii)] For every coclique $\rho \subseteq \pi(H)$ in $\Gamma(H)$ of size at least $3$, at most one prime in $\rho$ divides $|H|/|S| = |F(H)| \cdot |\overline{H}/S|$. In particular, $t(S) \geq t(H)-1$. 
\item[(iii)] Every prime $r \in \pi(H)$ that is not adjacent to $2$ in $\Gamma(H)$ does not divide $|H|/|S| = |F(H)| \cdot |\overline{H}/S|$. In particular, $t(2,S) \geq t(2,H)$. 
\end{itemize}
\end{lemma}

\begin{proof}
By Theorem~\ref{thm:tables}, $\Gamma(H)$ has at least three connected components, so Theorem~\ref{GKthm} implies that there exists a non-abelian simple group $S$ such that $\overline{H}$ is almost simple with socle $S$, and that $s(S) \geq s(G)$. 
Moreover, $S$ cannot be isomorphic to $G$ because $G$ is recognisable by its labelled prime graph (see Table~\ref{tabSummary}), i.e. if $S \cong G$ then $\Gamma(H)$ and $\Gamma(G)$ are actually equal, so $H \cong G$, in contradiction with our assumption. 
The final assertion of (i) holds because if $\Gamma(S)$ had more vertices or more edges than $\Gamma(G)$, then so would $\Gamma(H)$, given that $S$ is a subquotient of $H$, and so $\Gamma(H)$ would not be isomorphic to $\Gamma(G)$. 

Assertions (ii) and (iii) follow from Theorem~\ref{GKthmVasilev} and Remark~\ref{rem:S=T}. 
Note that $H$ satisfies the hypotheses of Theorem~\ref{GKthmVasilev} because $s(H) \geq 3$, i.e. every vertex belongs to a coclique of size at least $3$, so certainly $t(H) \geq 3$ and $t(H,2) \geq 3 > 2$. 
In particular, case (b) of Theorem~\ref{GKthmVasilev}(iii) is ruled out because $t(H,2) > 2$, so case (a) must hold. 
\end{proof}

Note that if $G$, $H$, and $S$ are groups as in Lemma~$\ref{lemma:obs}$, and if $\Gamma(S)$ has fewer vertices than $\Gamma(H)$, then every prime divisor $r$ of $|H|/|S|$ divides either $|F(H)|$ or $|\text{Out}(S)|$ (or both). 
The following lemma applies when at least one such prime divides $|F(H)|$.

\begin{lemma} \label{lemma:modules}
Let $G$, $H$, and $S$ be groups as in Lemma~$\ref{lemma:obs}$. 
If $r$ is a prime dividing $|F(H)|$, then $\overline{H} := H/F(H)$ admits a faithful irreducible module $V$ in characteristic $r$ with the following property: if there is an element of prime order $p$ in $S \leq \overline{H}$ that fixes a non-zero vector in $V$, then $p$ is adjacent to $r$ in $\Gamma(H)$.
\end{lemma}

\begin{proof}
Our argument is similar to those given in \cite[Section~1]{Kon2} and \cite[Section~2]{MBCo1}. 
Choose two neighbouring terms $K$ and $L$ of a chief series for $H$ such that $K < L \leq F(H)$ and $V = L/K$ is an elementary abelian $r$-group. 
Because $V$ is a minimal normal subgroup of $H/K$, it may be regarded as an irreducible module for $H/K$ in characteristic $r$. 
We prove below that the centraliser $C_{H/K}(V)$ of $V$ in $H/K$ is equal to $F(H)/K$, whence it follows that $V$ may be regarded as a {\em faithful} irreducible module for $(H/K)/(F(H)/K) \cong H/F(H) = \overline{H}$. 
If there exists $h \in S \leq \overline{H}$ of prime order $p$ that fixes some non-zero vector $v \in V$, then $h$ and $v$ commute in the affine group $V{:}\overline{H}$, so $\Gamma(H)$ contains the edge $\{p,r\}$. 

It remains to show that $C_{H/K}(V) = F(H)/K$. 
We first claim that $F(H/K) = F(H)/K$. 
We have $F(H/K) = A/K$ for some normal subgroup $A$ of $H$, and we are claiming that $A = F(H)$. 
Note that $F(H) \trianglelefteq A$ because $F(H)/K \trianglelefteq F(H/K)$, given that $F(H)/K$ is nilpotent and normal in $H/K$. 
Considering the quotient of $H/K$ by $F(H)/K$, we find~that
\[
A/F(H) \cong (A/K)/(F(H)/K) \trianglelefteq (H/K)/(F(H)/K) \cong H/F(H) = \overline{H}.
\]
In particular, $A/F(H)$ is a nilpotent normal subgroup of an almost simple group, so $A = F(H)$ as claimed. 
Next, note that $F(H)/K \trianglelefteq C_{H/K}(V)$ because $F(H)/K = F(H/K)$ centralises every minimal normal subgroup of $H/K$. 
Note also that $C_{H/K}(V)$ is normal in $H/K$. 
If the inclusion of $F(H)/K$ in $C_{H/K}(V)$ is strict, then $\overline{C} := C_{H/K}(V)/(F(H)/K)$ is a proper normal subgroup of $(H/K) / (F(H)/K) \cong H/F(H) = \overline{H}$. 
Therefore, $\overline{C}$ contains $S$, and so $r$ is adjacent to every prime divisor of $|S|$ in $\Gamma(H)$. 
If $\pi(H)=\pi(S)$, then it follows that $\Gamma(H)$ is connected, contradicting the fact that $\Gamma(H)$ is disconnected in all cases considered in Lemma~\ref{lemma:obs}. 
Otherwise, $\pi(H) \setminus \pi(S)$ is non-empty, and Lemma~\ref{lemma:obs}(iii) implies that every prime in this set is adjacent to the vertex $2$ in $\Gamma(H)$, so again $\Gamma(H)$ is connected and we have a contradiction. 
Therefore, $C_{H/K}(V) = F(H)/K$ as claimed.
\end{proof}

\begin{remark}
We take this opportunity to point out two (trivial) typos in \cite[Section~2]{MBCo1}. 
On p.~195, the definition ``$V = K/L$'' should be replaced by ``$V = L/K$''. 
On p.~196, the sentence ``Note that there are exactly 8 conjugacy classes of elements of odd order in $\text{Co}_1$ that intersect both $\text{Co}_3$ and $H$.'' should be replaced by ``Note that there are exactly 8 conjugacy classes of elements of odd order in $H$ that belong to $\text{Co}_1$-classes that also intersect $\text{Co}_3$.''. 
(These typos do not affect the correctness of the associated arguments.)
We also note that \cite[Table~1]{MBCo1} should be amended in light of Remark~\ref{rem:Ru}.
\end{remark}

Let us now proceed with the proof of Theorem~\ref{thm1} for the groups in \eqref{recogG}.

\subsection{The case $G = \textnormal{J}_4$} 
By \cite[Theorem~B]{Zavar}, $\text{J}_4$ is the only group whose prime graph has $6$ connected components, so $\text{J}_4$ is recognisable by the isomorphism type of its prime graph.


\subsection{The case $G = \mathbb{M}$} 
The prime graph of $\mathbb{M}$ has $|\pi(\mathbb{M})| = 15$ vertices and $s(\mathbb{M}) = 4$ connected components. 
Moreover, $t(\mathbb{M}) \geq 11$ because the vertices other than $2$, $3$, $5$, and~$7$ form a coclique. 
(A picture of $\Gamma(\mathbb{M})$ can be found in \cite[Fig.~5]{Zavar}.)

Suppose towards a contradiction that $H$ is a group not isomorphic to $\mathbb{M}$ such that $\Gamma(H) \cong \Gamma(\mathbb{M})$. 
By Lemma~\ref{lemma:obs}(i), there is a non-abelian simple group $S$ not isomorphic to $\mathbb{M}$ such that $S \trianglelefteq \overline{H} := H/F(H) \leq \text{Aut}(S)$, and in particular $s(S) \geq 4$ and $|\pi(S)| \leq 15$.  
It follows from Theorem~\ref{thm:tables} and Lemma~\ref{lemma:E8} that $S$ must be one of the groups
\[
\text{PSL}_3(4),\; {}^2\text{E}_6(2),\; \text{M}_{22},\; \text{J}_1, \; \text{O'N},\; \text{Ly},\; \text{Fi}_{24}',\; \text{J}_4, \text{ or } {}^2\text{B}_2(2^{2m+1}) \text{ for some } m \geq 1.
\]

Lemma~\ref{lemma:obs}(ii) implies that $t(S) \geq t(H)-1 = 10$. 
In particular, $\Gamma(S)$ must have at least $10$ vertices, which rules out all of the possibilities from the above list except $\text{J}_4$ and $^2\text{B}_2(2^{2m+1})$ (for certain $m$). 
Per Lemma~\ref{lemma:J4}, the prime graph of $\text{J}_4$ has exactly $10$ vertices, and they do not form a coclique, so $S \neq \text{J}_4$. 
On the other hand, $t({}^2\text{B}_2(2^{2m+1})) = 4 < 10$ for all $m \geq 1$ by Lemma~\ref{lemma:suz}, so $S \neq {}^2\text{B}_2(2^{2m+1})$ for any $m \geq 1$.

\subsection{The case $G = \textnormal{Fi}_{24}'$} \label{ss:Fi24}
The prime graph of $\text{Fi}_{24}'$ is shown below.
\begin{center}
\begin{tikzpicture}
\graph [no placement] { 
29[x=-2,y=0]; 
23[x=-1,y=0.5]; 17[x=-1,y=-0.5]; 
13[x=0,y=0.5]; 11[x=0,y=-0.5]; 
2[x=1,y=0.5]; 3[x=1,y=-0.5]; 
5[x=2,y=0.5]; 7[x=2,y=-0.5]; 
2 -- {3,5,7,11,13}; 
3 -- {5,7,11,13}; 
5 -- 7
};
\end{tikzpicture}
\end{center}
We have $|\pi(\text{Fi}_{24}')| = 9$, $s(\text{Fi}_{24}') = 4$, and $t(\text{Fi}_{24}') = 6$. 

Suppose towards a contradiction that $H$ is a group not isomorphic to $\text{Fi}_{24}'$ such that $\Gamma(H) \cong \Gamma(\text{Fi}_{24}')$. 
By Lemma~\ref{lemma:obs}(i), there is a non-abelian simple group $S$ not isomorphic to $\text{Fi}_{24}'$ such that $S \trianglelefteq \overline{H}:= H/F(H) \leq \text{Aut}(S)$, with $s(S) \geq 4$ and $|\pi(S)| \leq 9$.  
It follows from Theorem~\ref{thm:tables}, Lemma~\ref{lemma:J4}, and Lemma~\ref{lemma:E8} that $S$ must be one of the groups
\[
\text{PSL}_3(4),\; {}^2\text{E}_6(2),\; \text{M}_{22},\; \text{J}_1,\; \text{O'N},\; \text{Ly}, \text{ or } {}^2\text{B}_2(2^{2m+1}) \text{ for some } m \geq 1.
\]

Lemma~\ref{lemma:obs}(ii) implies that $t(S) \geq t(H)-1 = 5$, and in particular that $\Gamma(S)$ must have at least $5$ vertices. 
This implies that: $S$ cannot be $\text{PSL}_3(4)$ because $\Gamma(\text{PSL}_3(4))$ has only $4$ vertices; $S$ cannot be ${}^2\text{B}_2(2^{2m+1})$ for any $m \geq 1$ because then $t(S)=4<5$ by Lemma~\ref{lemma:suz}; and $S$ cannot be either $\text{M}_{22}$ or $\text{J}_1$ because in both of these cases $\Gamma(S)$ consists of three isolated vertices and a clique (cf. Section~\ref{ss:J1M22}), so $t(S) = 4 < 5$. 

It remains to show that $S$ cannot be any of ${}^2\text{E}_6(2)$, $\text{O'N}$, or $\text{Ly}$. 
We do this using Lemma~\ref{lemma:modules}. 
The prime graphs of $\text{Ly}$ and $\text{O'N}$ are shown in Sections~\ref{ss:Ly} and \ref{ss:O'N}, and the prime graph of ${}^2\text{E}_6(2)$ is shown below.
\begin{center}
\begin{tikzpicture}
\graph [no placement] { 
19[x=-2,y=0]; 
17[x=-1,y=0.5]; 13[x=-1,y=-0.5]; 
2[x=1,y=0.5]; 3[x=1,y=-0.5]; 
11[x=0,y=0.5]; 5[x=2,y=0.5]; 7[x=2,y=-0.5];
2 -- {3,5,7,11}; 
3 -- {5,7,11}; 
5 -- 7
};
\end{tikzpicture}
\end{center}
In each case, $\Gamma(S)$ has fewer vertices than $\Gamma(H) \cong \Gamma(\text{Fi}_{24}')$, and $|\text{Out}(S)| \in \{1,2,6\}$ divides $|S|$, so there must be a prime $r \in \pi(H) \setminus \pi(S)$ dividing $|F(H)|$. 
Lemma~\ref{lemma:modules} therefore implies that there must be a faithful irreducible module $V$ for $\overline{H}$ in characteristic $r$ such that there exists an element of prime order $p$ in $S \leq \overline{H}$ fixing a non-zero vector in $V$ only if $p$ is adjacent to $r$ in $\Gamma(H)$. 
In particular, $S$ itself must admit such a module.

If $S = {}^2\text{E}_6(2)$, then $r$ cannot be adjacent to $13$, $17$, or $19$, because $\Gamma(H)$ would then have at most two isolated vertices and hence not be isomorphic to $\Gamma(\text{Fi}_{24}')$. 
Because $r$ is coprime to $|S|$, we can use the ordinary character table of $S$, which is available in the Atlas~\cite{ATLAS} and in \textsf{GAP}~\cite{GAPbc,GAP}, to check that this is impossible: every element of order $13$, $17$, or $19$ in $S$ fixes a non-zero vector in every faithful irreducible module for $S$ in characteristic coprime to $|S|$. 
As previously noted, this can be verified using Lemma~\ref{lemma:brauer}. 
Similarly, if $S = \text{O'N}$ then $r$ cannot be adjacent to $11$, $19$, or $31$, and if $S = \text{Ly}$ then $r$ cannot be adjacent to $31$, $37$, or $67$. 
In each case, inspecting the character table of $S$ yields a contradiction.

\subsection{The case $G = \textnormal{Ly}$} \label{ss:Ly}
The prime graph of $\text{Ly}$ is shown below.
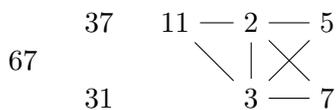
\begin{center}
\begin{tikzpicture}
\graph [no placement] { 
67[x=-2,y=0]; 
37[x=-1,y=0.5]; 31[x=-1,y=-0.5]; 
2[x=1,y=0.5]; 3[x=1,y=-0.5]; 
11[x=0,y=0.5]; 5[x=2,y=0.5]; 7[x=2,y=-0.5];
2 -- {3,5,7,11}; 
3 -- {5,7,11}
};
\end{tikzpicture}
\end{center}
We have $|\pi(\text{Ly})| = 8$, $s(\text{Ly}) = 4$, and $t(\text{Ly}) = 6$. 

Suppose towards a contradiction that $H$ is a group not isomorphic to $\text{Ly}$ such that $\Gamma(H) \cong \Gamma(\text{Ly})$. 
By Lemma~\ref{lemma:obs}(i), there is a non-abelian simple group $S$ not isomorphic to $\text{Ly}$ such that $S \trianglelefteq \overline{H} := H/F(H) \leq \text{Aut}(S)$, with $s(S) \geq 4$ and $|\pi(S)| \leq 8$.  
It follows from Theorem~\ref{thm:tables}, Lemma~\ref{lemma:J4}, and Lemma~\ref{lemma:E8} that $S$ must be one of
\[
\text{PSL}_3(4),\; {}^2\text{E}_6(2),\; \text{M}_{22},\; \text{J}_1,\; \text{O'N}, \text{ or } {}^2\text{B}_2(2^{2m+1}) \text{ for some } m \geq 1,
\]
and Lemma~\ref{lemma:obs}(ii) implies that $t(S) \geq t(H)-1 = 5$, so the same arguments as in the case $G = \text{Fi}_{24}'$ show that $S \neq \text{PSL}_3(4)$, $\text{M}_{22}$, $\text{J}_1$, or ${}^2\text{B}_2(2^{2m+1})$ for any $m \geq 1$. 
Moreover, $S \neq {}^2\text{E}_6(2)$ because $\Gamma({}^2\text{E}_6(2))$, which is shown in Section~\ref{ss:Fi24}, has more edges than $\Gamma(\text{Ly})$. 

Finally, if $S = \text{O'N}$ then $\Gamma(S)$ (which is shown in Section~\ref{ss:O'N}) has fewer vertices than $\Gamma(H) \cong \Gamma(\text{O'N})$, and $|\text{Out}(S)|=2$ divides $|S|$ so there exists a prime $r \in \pi(H) \setminus \pi(S)$ dividing $|F(H)|$. 
Lemma~\ref{lemma:modules} therefore implies that there is a faithful irreducible module $V$ for $S$ in characteristic $r$ such that there exists an element of prime order $p$ in $S$ fixing a non-zero vector in $V$ only if $p$ is adjacent to $r$ in $\Gamma(H)$. 
We see that $r$ cannot be adjacent to $11$, $19$, or $31$, or else $\Gamma(H)$ would have at most two isolated vertices, whereas $\Gamma(\text{O'N})$ has three. 
Therefore, the same argument as in Section~\ref{ss:Fi24} yields a contradiction.

\subsection{The case $G = \textnormal{O'N}$} \label{ss:O'N}
The prime graph of $\text{O'N}$ is shown below.
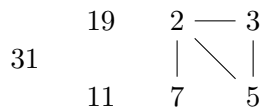
\begin{center}
\begin{tikzpicture}
\graph [no placement] { 
31[x=-2,y=0]; 
19[x=-1,y=0.5]; 11[x=-1,y=-0.5]; 
2[x=0,y=0.5]; 7[x=0,y=-0.5]; 
3[x=1,y=0.5]; 5[x=1,y=-0.5];
2 -- {3,5,7}; 
3 -- 5
};
\end{tikzpicture}
\end{center}
We have $|\pi(\text{O'N})| = 7$, $s(\text{O'N}) = 4$, and $t(\text{O'N}) = 5$. 

Suppose towards a contradiction that $H$ is a group not isomorphic to $\text{O'N}$ such that $\Gamma(H) \cong \Gamma(\text{O'N})$. 
By Lemma~\ref{lemma:obs}(i), there is a non-abelian simple group $S$ not isomorphic to $\text{O'N}$ such that $S \trianglelefteq \overline{H} := H/F(H) \leq \text{Aut}(S)$, with $s(S) \geq 4$ and $|\pi(S)| \leq 7$.  
It follows from Theorem~\ref{thm:tables}, Lemma~\ref{lemma:J4}, and Lemma~\ref{lemma:E8} that $S$ must be one of
\[
\text{PSL}_3(4),\; \text{M}_{22},\; \text{J}_1, \text{ or } {}^2\text{B}_2(2^{2m+1}) \text{ for some } m \geq 1.
\]

If $S = \text{PSL}_3(4)$, then $\Gamma(S)$ consists of the four isolated vertices $2$, $3$, $5$, and $7$. 
Moreover, $|\text{Out}(S)| = 12$ divides $|S|$, so there must be $|\pi(H) \setminus \pi(S)| = 3$ primes dividing $|F(H)|$ but not $|S|$. 
These three primes form a clique in $\Gamma(H)$ because $F(H)$ is a direct product of its Sylow subgroups. 
By Lemma~\ref{lemma:obs}(iii), they are also all adjacent to the vertex $2$, and so we obtain a clique of size $4$ in $\Gamma(H)$, contradicting the fact that $\Gamma(\text{O'N})$ has no such clique. 

If $S = \text{M}_{22}$, then $\Gamma(S)$ consists of the three isolated vertices $5$, $7$, and $11$, and the edge $\{2,3\}$. 
Moreover, $|\text{Out}(S)| = 2$ divides $|S|$, so there must be $|\pi(H) \setminus \pi(S)| = 2$ primes $r_1$ and $r_2$ dividing $|F(H)|$ but not $|S|$, which form an edge in $\Gamma(H)$. 
By Lemma~\ref{lemma:obs}(iii), both $r_1$ and $r_2$ are adjacent to $2$ in $\Gamma(H)$. 
Comparing with $\Gamma(\text{O'N})$, we see that there can be no other edges in $\Gamma(H)$. 
In particular, by Lemma~\ref{lemma:modules}, elements of order $3$ in $S$ must act fixed-point freely on the non-zero vectors of some faithful irreducible module for $S$ in characteristic coprime to $|S|$. 
The character table of $S$ shows that no such module exists. 

If $S = \text{J}_1$, then $\Gamma(S)$ consists of the three isolated vertices $7$, $11$, and $19$, and the clique $\{2,3,5\}$. 
Moreover, $\text{Out}(S)$ is trivial, so there must be $|\pi(H) \setminus \pi(S)| = 1$ prime $r$ dividing $|F(H)|$ but not $|S|$. 
This prime is adjacent to $2$ in $\Gamma(H)$ by Lemma~\ref{lemma:obs}(iii), and there can be no other edges in $\Gamma(H)$. 
In particular, by Lemma~\ref{lemma:modules}, elements of order $3$ in $S$ must act fixed-point freely on the non-zero vectors of some faithful irreducible module for $S$ in characteristic coprime to $|S|$. 
However, no such module exists. 

Finally, suppose that $S = {}^2\text{B}_2(2^{2m+1})$ for some $m \geq 1$.
By Lemma~\ref{lemma:suz}, $\Gamma(S)$ consists of the isolated vertex $2$ and three cliques, which we denote by $C_1$, $C_2$, and $C_3$. 
Given that $\Gamma(\text{O'N})$ has $7$ vertices, we have $3 \leq N \leq 6$, where $N$ is the total number of vertices in $C_1$, $C_2$, and $C_3$. 
If $N=6$, then $\pi(H) = \pi(S)$ and, by Lemma~\ref{lemma:obs}(iii), two of the $C_i$ must be isolated vertices given that $\Gamma(\text{O'N})$ has three isolated vertices. 
The third $C_i$ is then a clique of size $4$, contradicting the fact that $\Gamma(\text{O'N})$ has no such clique. 
If $N \in \{4,5\}$, then there is at least one prime in $\pi(H) \setminus \pi(S)$, and this prime is adjacent to $2$ in $\Gamma(H)$ by Lemma~\ref{lemma:obs}(iii). 
On the other hand, at most two of the $C_i$ have size $1$, so $\Gamma(H)$ has at most two isolated vertices, a contradiction. 
Therefore, $N=3$, i.e. all of the $C_i$ are isolated vertices. 
In particular, $m=1$ or $2$ by Lemma~\ref{lemma:suz}. 
There are now three primes $r_1,r_2,r_3 \in \pi(H) \setminus \pi(S)$, all adjacent to $2$ in $\Gamma(H)$ by Lemma~\ref{lemma:obs}(iii). 
If they all divide $|F(H)|$, then $\{2,r_1,r_2,r_3\}$ is a clique, a contradiction. 
Therefore, $\overline{H} = H/F(H)$ must properly contain $S$, and $r_1$ (say) must divide $|\text{Out}(S)|$ but not $|S|$ or $|F(H)|$. 
Because $m=1$ or $2$, $\text{Out}(S)$ is cyclic of prime order $2m+1=3$ or $5$. 
The latter case yields a contradiction because $5 \in \pi(S)$. 
Therefore, $m=1$, the vertex set of $\Gamma(S)$ is $\{2,5,7,13\}$, $\overline{H}=\text{Aut}(S)$, and $r_2$ and $r_3$ divide $|F(H)|$, yielding the edge $\{r_2,r_3\}$ in $\Gamma(H)$. 
If $\sigma$ is a generator of $\text{Out}(S)$, then $\sigma$ centralises a subgroup ${}^2\text{B}_2(2) \cong 5{:}4$ of $S$ (cf. Section~\ref{ss:M23}), so $\Gamma$ also contains the edge $\{r_1,5\} = \{3,5\}$. 
This means that $\Gamma(H)$ has at most two isolated vertices, $7$ and $13$, a contradiction.


\subsection{The case $G = \mathbb{B}$} \label{ss:B}
The prime graph of $\mathbb{B}$ is shown below.
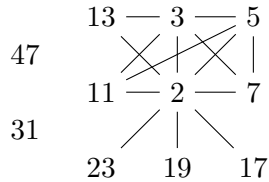
\begin{center}
\begin{tikzpicture}
\graph [no placement] { 
47[x=-1,y=0]; 31[x=-1,y=-1]; 
13[x=0,y=0.5]; 3[x=1,y=0.5]; 5[x=2,y=0.5];
11[x=0,y=-0.5]; 2[x=1,y=-0.5]; 7[x=2,y=-0.5]; 
23[x=0,y=-1.5]; 19[x=1,y=-1.5]; 17[x=2,y=-1.5]; 
2 -- {3,5,7,11,13,17,19,23}; 
3 -- {5,7,11,13}; 
5 -- {7,11}
};
\end{tikzpicture}
\end{center}
We have $|\pi(\mathbb{B})| = 11$, $s(\mathbb{B}) = 3$, and $t(\mathbb{B}) = 8$. 

Suppose towards a contradiction that $H$ is a group not isomorphic to $\mathbb{B}$ such that $\Gamma(H) \cong \Gamma(\mathbb{B})$. 
By Lemma~\ref{lemma:obs}(i), there is a non-abelian simple group $S$ not isomorphic to $\mathbb{B}$ such that $S \trianglelefteq \overline{H} := H/F(H) \leq \text{Aut}(S)$, with $s(S) \geq 3$ and $|\pi(S)| \leq 11$. 
Moreover, Lemma~\ref{lemma:obs}(ii) implies that we must have $t(S) \geq t(H)-1 \geq 7$, so in particular $|\pi(S)| \geq 7$.

\subsubsection{{\bf The sub-case $s(S) \geq 4$ for $G = \mathbb{B}$}}
We first rule out the candidates for $S$ with $s(S) \geq 4$. 
By Theorem~\ref{thm:tables}, Lemma~\ref{lemma:J4}, and Lemma~\ref{lemma:E8}, $S$ must be one of
\[
{}^2\text{E}_6(2),\; \text{Fi}_{24}',\; \text{J}_4,\; \text{Ly},\; \text{O'N}, \text{ or } {}^2\text{B}_2(2^{2m+1}) \text{ for some } m \geq 1, 
\]
i.e. all other simple groups $S$ with $s(S) \geq 4$ have $|\pi(S)| > 11$ or $|\pi(S)| < 7$. 
In each case except $S = \text{J}_4$, we have $t(S) < 7$. 
(For ${}^2\text{B}_2(2^{2m+1})$, see Lemma~\ref{lemma:suz}; for the other four groups, including ${}^2\text{E}_6(2)$, refer to the prime graphs shown in Sections \ref{ss:Fi24}, \ref{ss:Ly} and \ref{ss:O'N}.) 
If $S = \text{J}_4$, then $\text{Out}(S)$ is trivial, so there must be $|\pi(H) \setminus \pi(S)| = 1$ prime $r$ dividing $|F(H)|$ but not $|S|$. 
The character table of $S$ shows, in particular, that for every isolated vertex $p$ in $\Gamma(S)$, every element of order $p$ in $S$ fixes a non-zero vector in every faithful irreducible $r$-modular representation of $S$. 
Lemma~\ref{lemma:modules} therefore implies that $\Gamma(H)$ is connected.

\subsubsection{{\bf The sub-case $s(S) = 3$ for $G = \mathbb{B}$}} \label{ss:Bs=3}
We now rule out the candidates for $S$ with $s(S) = 3$, which are listed in Table~\ref{tab:s(S)=3}. 
The constraint $7 \leq |\pi(S)| \leq 11$ eliminates the groups $\text{PSU}_6(2)$, $\text{E}_7(2)$, and $\text{E}_7(3)$, and all sporadic groups except $S=\text{Th}$ and $S=\text{Fi}_{23}$. 
The prime graphs of these two groups are shown in Sections~\ref{ss:Th} and~\ref{ss:Fi23}. 
In both cases, $t(S) = 5 < 7$. 
The constraint $t(S) \geq 7$ also rules out all of the groups $\text{PSL}_2(q)$, $\text{G}_2(3^k)$, ${}^2\text{G}_2(3^{2m+1})$, and ${}^2\text{F}_4(2^{2m+1})$, by Lemmas~\ref{lemma:A1}, \ref{lemma:G2}, \ref{lemma:2F4}, and \ref{lemma:2G2_2}.

If $S = \text{Alt}_p$ with $p > 6$ and $p-2$ both prime, then $7 \leq |\pi(S)| \leq 11$ forces $p \in \{19,31\}$. 
In both cases, $\Gamma(S)$ contains a clique $\{2,3,5,7,11\}$ of size $5$, but $\Gamma(\mathbb{B})$ has no such clique. 

If $S = \text{P}\Omega_{2p}^-(3)$ with $p$ a prime of the form $2^m+1$, $m \geq 1$, then Lemma~\ref{lemma:2Dp} tells us that $|\pi(S)| \geq p+1$. 
Given that we need $|\pi(S)| \leq 11$, the only possibilities are $p = 3$ and $p = 5$. 
If $p=3$, then $S \cong \text{PSU}_4(3)$ so $|\pi(S)| = 4 < 7$. 
If $p=5$, then $\Gamma(S)$ is the graph
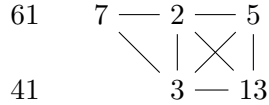
\begin{center}
\begin{tikzpicture}
\graph [no placement] { 
61[x=-1,y=0.5]; 41[x=-1,y=-0.5]; 
2[x=1,y=0.5]; 3[x=1,y=-0.5]; 
7[x=0,y=0.5]; 5[x=2,y=0.5]; 13[x=2,y=-0.5];
2 -- {3,5,7,13}; 
3 -- {5,7,13}; 
5 -- 13
};
\end{tikzpicture}
\end{center}
and we have $t(S) = 5 < 7$. 

Finally, suppose that $S = \text{F}_4(q)$ with $q$ even. 
If $q \geq 4$, then Lemma~\ref{lemma:F4} tells us that $\Gamma(S)$ has at least three vertices of degree at least $5$. 
However, $\Gamma(\mathbb{B})$ has only two such vertices, namely $2$ and $3$. 
If $q=2$, then $\Gamma(S) = \Gamma(\text{F}_4(2))$ is the graph
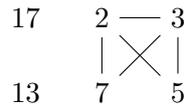
\begin{center}
\begin{tikzpicture}
\graph [no placement] { 17[x=0,y=1]; 13[x=0,y=0]; 2[x=1,y=1];  {3[x=2,y,=1],5[x=2,y=0]} -- 2; 7[x=1,y=0] -- {2,3} ; 3 -- 5 };
\end{tikzpicture}
\end{center}
and we have $t(S) = 4 < 7$.


\subsection{The case $G = \textnormal{Th}$} \label{ss:Th}
The prime graph of $\text{Th}$ is shown below.
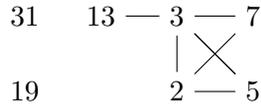
\begin{center}
\begin{tikzpicture}
\graph [no placement] { 
31[x=0,y=1]; 
19[x=0,y=0]; 
3[x=2,y=1]; 
{7[x=3,y,=1],5[x=3,y=0]} -- 3; 
2[x=2,y=0] -- {3,7}; 
3 -- 5; 
2 -- 5;
13[x=1,y=1] -- 3 
};
\end{tikzpicture}
\end{center}
We have $|\pi(\text{Th})| = 7$, $s(\text{Th}) = 3$, and $t(\text{Th}) = 5$. 

Suppose towards a contradiction that $H$ is a group not isomorphic to $\text{Th}$ such that $\Gamma(H) \cong \Gamma(\text{Th})$. 
By Lemma~\ref{lemma:obs}(i), there is a non-abelian simple group $S$ not isomorphic to $\text{Th}$ such that $S \trianglelefteq \overline{H} := H/F(H) \leq \text{Aut}(S)$, with $s(S) \geq 3$ and $|\pi(S)| \leq 7$.

\subsubsection{{\bf The sub-case $s(S) \geq 4$ for $G = \textnormal{Th}$}}
We begin by ruling out the candidates for $S$ with $s(S)\geq 4$. 
By Theorem~\ref{thm:tables}, Lemma~\ref{lemma:J4}, and Lemma~\ref{lemma:E8}, $S$ must be one of
\[
\text{PSL}_3(4),\; \text{M}_{22},\; \text{J}_1,\; \text{O'N},\text{ or } {}^2\text{B}_2(2^{2m+1}) \text{ for some } m \geq 1, 
\]
i.e. all other simple groups $S$ with $s(S) \geq 4$ have $|\pi(S)| > 7$. 

If $S = \text{PSL}_3(4)$, then $\Gamma(S)$ consists of the four isolated vertices $2$, $3$, $5$, and $7$. 
Because $|\text{Out}(S)| = 12$ divides $|S|$, there must be $|\pi(H) \setminus \pi(S)| = 3$ primes dividing $|F(H)|$ but not $|S|$. 
These primes form a clique in $\Gamma(H)$, and they are all adjacent to $2$ by Lemma~\ref{lemma:obs}(iii), yielding a clique of size $4$. 
However, $\Gamma(\text{Th})$ has no such clique. 

If $S = \text{M}_{22}$, then $\Gamma(S)$ consists of the three isolated vertices $5$, $7$, and $11$, and the edge $\{2,3\}$. 
Because $|\text{Out}(S)| = 2$ divides $|S|$, there must be $|\pi(H) \setminus \pi(S)| = 2$ primes dividing $|F(H)|$ but not $|S|$. 
These primes form an edge in $\Gamma(H)$, and both are adjacent to $2$ by Lemma~\ref{lemma:obs}(iii). 
Lemma~\ref{lemma:modules} implies that they are also both adjacent to $3$, because elements of order $3$ fix a non-zero vector in every representation of $S$ in characteristic coprime to $|S|$. 
Therefore, $\Gamma(H)$ contains a clique of size $4$, whereas $\Gamma(\text{Th})$ has no such clique.

If $S = \text{J}_1$, then $\Gamma(S)$ consists of the three isolated vertices $7$, $11$, and $19$, and the clique $\{2,3,5\}$. 
Because $\text{Out}(S)$ is trivial, there must be $|\pi(H) \setminus \pi(S)| = 1$ prime dividing $|F(H)|$ but not $|S|$. 
This prime is adjacent to $2$ by Lemma~\ref{lemma:obs}(iii). 
Because $\Gamma(H)$ must have two isolated vertices, Lemma~\ref{lemma:modules} implies that must $S$ admit a faithful irreducible module $V$ in characteristic coprime to $|S|$ such that elements of order $p$ fix a non-zero vector in $V$ for at most one $p \in \{7,11,19\}$. 
The character table of $S$ shows that no such module exists.

If $S = \text{O'N}$, then $\Gamma(S)$ is shown in Section~\ref{ss:O'N}. 
Note that $|\pi(S)| = |\pi(\text{Th})|$, and that $\Gamma(\text{Th})$ has two more edges than $\Gamma(S)$. 
The outer automorphism group of $S$ has order $2$, and if $\sigma$ is a (non-trivial) outer automorphism then $|C_{\text{Aut}(S)}(\sigma)|$ is divisible by every prime divisor of $|S|$ other than $31$. 
Therefore, $\overline{H} = H/F(H)$ cannot contain any outer automorphisms of $S$, or else $\Gamma(\overline{H})$ and hence $\Gamma(H)$ would have at most two connected components. 
Hence, $F(H)$ must be non-trivial, with $\pi(F(H)) \subseteq \pi(S)$ because $|\pi(S)| = |\pi(\text{Th})|$. 
Given that $\Gamma(H)$ must have two isolated vertices, there must be exactly one edge joining the sets $X_1 = \{2,3,5,7\}$ and $X_2 = \{11,19,31\}$. 
Let $r \in F(H)$. 
If $r \in X_2$, then $\{r,2\}$ is an edge in $\Gamma(H)$ by Lemma~\ref{lemma:obs}(iii), and the $r$-modular character table of $S$ shows, in particular, that every element of order $3$ in $S$ fixes a non-zero vector in every faithful irreducible $r$-modular representation of $S$, so Lemma~\ref{lemma:modules} implies that $\{r,3\}$ is also an edge. 
Hence, there is more than one edge joining $X_1$ and $X_2$, a contradiction. 
If $r \in X_1$, then the $r$-modular character table of $S$ shows, in particular, that every element of order $11$ or $31$ in $S$ fixes a non-zero vector in every faithful irreducible $r$-modular representation of $S$, so $\{r,11\}$ and $\{r,31\}$ are edges in $\Gamma(H)$, a contradiction. 
(Note that all of the required Brauer character tables of $\text{O'N}$ are available in {\sf GAP}.)

Now suppose that $S = {}^2\text{B}_2(2^{2m+1})$ for some $m \geq 1$.
By Lemma~\ref{lemma:suz}, $\Gamma(S)$ consists of the isolated vertex $2$ and three cliques $C_1$, $C_2$, and $C_3$. 
Given that $|\pi(\text{Th})|=7$, we have $3 \leq N \leq 6$, where $N$ is the total number of vertices in $C_1$, $C_2$, and $C_3$.  
If $N=6$, then $\pi(H) = \pi(S)$ and Lemma~\ref{lemma:obs}(iii) implies that exactly one of the $C_i$ must be an isolated vertex, because $\Gamma(\text{Th})$ has two isolated vertices and no cliques of size $4$. 
The other two $C_i$ must have sizes $2$ and $3$. 
However, this is impossible because $\Gamma(\text{Th})$ does not contain a subgraph consisting of a disjoint triangle and edge. 
If $N=5$, then there is one prime $r \in \pi(H) \setminus \pi(S)$, and $\{2,r\}$ is an edge in $\Gamma(H)$ by Lemma~\ref{lemma:obs}(iii). 
Two of the $C_i$ must be isolated vertices, and the third $C_i$ must be a triangle, so $\Gamma(H)$ again contains a forbidden subgraph consisting of a disjoint triangle and edge. 
If $N=4$, then there are two primes $r_1,r_2 \in \pi(H) \setminus \pi(S)$, and both are adjacent to $2$. 
Therefore, $C_1$ and $C_2$ must be isolated vertices, and $C_3$ must be an edge. 
If either of the $r_i$ divides $|F(H)|$, then Lemmas~\ref{lemma:suz_reps} and~\ref{lemma:modules} imply that $\Gamma(H)$ is connected, so both of the $r_i$ must divide $|\text{Out(S)}|$ but not $|F(H)|$ or $|S|$. 
In particular, $\{r_1,r_2\}$ must be an edge because $\text{Out}(S)$ is cyclic, so we obtain the same forbidden subgraph as for $N \in \{5,6\}$. 
Finally, if $N=3$, then all of the $C_i$ are isolated vertices, so $m\in\{1,2\}$ by Lemma~\ref{lemma:suz}.  
There are now three distinct primes $r_1,r_2,r_3 \in \pi(H) \setminus \pi(S)$, all adjacent to~$2$. 
None of these primes can divide $|F(H)|$, or else $\Gamma(H)$ would be connected (as above), so $|\text{Out}(S)| = 2m+1 \in \{3,5\}$ must be divisible by three distinct primes, but it is not.

\subsubsection{{\bf The sub-case $s(S) = 3$ for $G = \textnormal{Th}$}} \label{ss:Ths=3}
We now rule out the candidates for $S$ with $s(S) = 3$. 
By Theorem~\ref{thm:tables}, we must consider all $S$ in Table~\ref{tab:s(S)=3} with $|\pi(S)| \leq |\pi(\text{Th})| = 7$ (other than $\text{Th}$ itself). 
Because $\Gamma(\text{Th})$ contains a coclique of size $5$, Lemma~\ref{lemma:obs}(ii) further implies that $\Gamma(S)$ must contain a coclique of size $4$, i.e. that $t(S) \geq 4$. 

We start with the sporadic groups, noting that we can ignore $\mathbb{B}$ and $\text{Fi}_{23}$ because their prime graphs have more than $7$ vertices.
Moreover, if $S = \textnormal{HS}$, $\textnormal{J}_{3}$, or $\text{M}_{11}$, then $\Gamma(S)$ consists of three cliques (cf. Sections~\ref{ss:M11} and~\ref{ss:HS-J3}), so $t(S) = 3 < 4$.  

If $S = \text{Suz}$ or $\text{M}_{24}$, then, per Section~\ref{ss:M24-Suz}, the prime graph of $S$ is
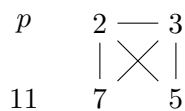
\begin{center}
\begin{tikzpicture}
\graph [no placement] { $p$[x=0,y=1]; 11[x=0,y=0]; 2[x=1,y=1];  {3[x=2,y,=1],5[x=2,y=0]} -- 2; 7[x=1,y=0] -- {2,3} ; 3 -- 5 };
\end{tikzpicture}
\end{center}
where $p=11$ or $23$ respectively. 
We have $|\text{Out}(S)| \leq 2$ in both cases, so there must be $|\pi(H) \setminus \pi(S)| = 1$ prime dividing $|F(H)|$ but not $|S|$. 
This prime is adjacent to $2$ in $\Gamma(H)$ by Lemma~\ref{lemma:obs}(iii). 
Given that $\Gamma(H)$ must have two isolated vertices, Lemma~\ref{lemma:modules} implies that $S$ must admit a faithful irreducible module $V$ in characteristic coprime to $|S|$ such that elements of order $11$ fix no non-zero vector in $V$. 
However, no such module exists.

If $S = \text{Co}_2$ or $\text{M}_{23}$, then $\Gamma(S)$ is obtained from $\Gamma(\text{M}_{24})$ by deleting the edge $\{3,7\}$ in both cases and also deleting the edge $\{3,5\}$ in the case $S = \text{M}_{23}$. 
Because $\text{Out}(S)$ is trivial in both cases, there must again be a prime $r$ dividing $|F(H)|$ but not $|S|$. 
The character table of $S$ shows that every element of order $11$ in $S$ fixes a non-zero vector in every faithful irreducible representation of $S$ in characteristic coprime to $|S|$, so Lemma~\ref{lemma:modules} implies that $\Gamma(H)$ contains the edge $\{11,r\}$. 
In particular, $\Gamma(H)$ does not have two isolated vertices. 

It remains to deal with the non-sporadic groups in Table~\ref{tab:s(S)=3}. 
We can ignore the groups $\text{E}_7(2)$ and $\text{E}_7(3)$ because their prime graphs have more than $|\pi(\text{Th})| = 7$ vertices. 
By Lemma~\ref{lemma:2F4}, the same can be said of the groups ${}^2\text{F}_4(2^{2m+1})$, $m \geq 1$. 
The groups $\text{PSL}_2(q)$ and $\text{G}_2(3^k)$ are ruled out because their prime graphs consist of three cliques, i.e. because $t(S)=3<4$ (see Lemmas~\ref{lemma:A1} and~\ref{lemma:G2}). 
The case $S=\text{PSU}_6(2)$ is likewise eliminated upon noting that the connected component $\{2,3,5\}$ of $\Gamma(S)$ is a clique. 

If $S = \text{Alt}_p$ for some prime $p>6$ such that $p-2$ is also prime, then $|\pi(S)| \leq 7$ forces $p \in \{7,13\}$. 
In both cases, $\Gamma(S)$ consists of three cliques, so $t(S)=3<4$. 

Now suppose that $S = \text{P}\Omega_{2p}^-(3)$ with $p = 2^m+1$ a prime, $m \geq 1$. 
We need $|\pi(S)| \leq 7$, so Lemma~\ref{lemma:2Dp} implies that $p \in \{3,5\}$. 
If $p=3$, then $S \cong \text{PSU}_4(3)$ and $\Gamma(S)$ consists of three cliques (the edge $\{2,3\}$ and the isolated vertices $5$ and $7$), so $t(S) = 3 < 4$. 
If $p=5$, then $\Gamma(S)$ contains a clique of size $4$, namely $\{2,3,5,7\}$, but $\Gamma(\text{Th})$ has no such clique. 
(The graph $\Gamma(\text{P}\Omega_{10}^-(3))$ is shown in Section~\ref{ss:Bs=3}.)

If $S=\text{F}_4(q)$ with $q$ even, then Lemma~\ref{lemma:F4} and the constraint $|\pi(S)| \leq 7$ imply that $q=2$. 
The prime graph of $S=\text{F}_4(2)$ is shown in Section~\ref{ss:Bs=3}. 
Because $|\text{Out}(S)| = 2$, there must be $|\pi(H) \setminus \pi(S)| = 1$ prime dividing $|F(H)|$ but not $|S|$. 
Given that $\Gamma(H)$ must have two isolated vertices, Lemma~\ref{lemma:modules} implies that $S$ must admit a faithful irreducible module $V$ in characteristic coprime to $|S|$ such that elements of order $13$ fix no non-zero vector in $V$. 
The character table of $S$ indicates that no such module exists.

It remains to deal with the case $S = {}^2\text{G}_2(3^{2m+1})$, $m \geq 1$. 
First suppose that $m \geq 2$. 
Given that $\Gamma(H)$ must have two isolated vertices, its subgraph $\Gamma(S)$ must have two isolated vertices by Lemma~\ref{lemma:obs}(iii). 
Lemma~\ref{lemma:2G2_2} therefore implies that $|\pi(S)| \geq 9 > |\pi(\text{Th})|$. 
If $m=1$, then $\Gamma(S)$ consists of the vertex set $\{2,3,7,13,19,37\}$ and the edges $\{2,3\}$, $\{2,7\}$, and $\{2,13\}$. 
Because $|\text{Out}(S)|=3$ divides $|S|$, there must be $|\pi(H) \setminus \pi(S)|=2$ primes $r_1$ and $r_2$ dividing $|F(H)|$ but not $|S|$. 
The character table of $S$ shows, in particular, that every element of order $19$ in $S$ fixes a non-zero vector in every faithful irreducible representation of $S$ in characteristic coprime to $|S|$. 
Lemma~\ref{lemma:modules} therefore implies that $\Gamma(H)$ contains the edges $\{19,r_1\}$ and $\{19,r_2\}$, so $\Gamma(H)$ does not have two isolated vertices.


\subsection{The case $G = \textnormal{Fi}_{23}$} \label{ss:Fi23}
The prime graph of $\text{Fi}_{23}$ is shown below.
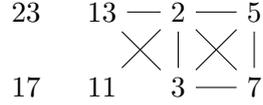
\begin{center}
\begin{tikzpicture}
\graph [no placement] { 
23[x=-1,y=0.5]; 17[x=-1,y=-0.5]; 
13[x=0,y=0.5]; 11[x=0,y=-0.5]; 
2[x=1,y=0.5]; 3[x=1,y=-0.5]; 
5[x=2,y=0.5]; 7[x=2,y=-0.5]; 
2 -- {3,5,7,11,13}; 
3 -- {5,7,13}; 
5 -- 7
};
\end{tikzpicture}
\end{center}
We have $|\pi(\text{Fi}_{23})| = 8$, $s(\text{Fi}_{23}) = 3$, and $t(\text{Fi}_{23}) = 5$. 
Let us also record the following observations about $\Gamma(\text{Fi}_{23})$:
\begin{itemize}
\item[(O1)] $\Gamma(\text{Fi}_{23})$ has a unique clique of size $4$, and every edge (respectively triangle) in $\Gamma(\text{Fi}_{23})$ has at least one vertex (respectively edge) in common with this clique;
\item[(O2)] every pair of triangles in $\Gamma(\text{Fi}_{23})$ has at least one vertex in common. 
\end{itemize}

Suppose towards a contradiction that $H$ is a group not isomorphic to $\text{Fi}_{23}$ such that $\Gamma(H) \cong \Gamma(\text{Fi}_{23})$. 
By Lemma~\ref{lemma:obs}(i), there is a non-abelian simple group $S$ not isomorphic to $\text{Fi}_{23}$ such that $S \trianglelefteq \overline{H} := H/F(H) \leq \text{Aut}(S)$, with $s(S) \geq 3$ and $|\pi(S)| \leq 8$.

\subsubsection{{\bf The sub-case $s(S) \geq 4$ for $G = \textnormal{Fi}_{23}$}}
We first rule out the candidates for $S$ with $s(S) \geq 4$. 
By Theorem~\ref{thm:tables}, Lemma~\ref{lemma:J4}, and Lemma~\ref{lemma:E8}, $S$ must be one of
\[
\text{PSL}_3(4),\; {}^2\text{E}_6(2),\; \text{M}_{22},\; \text{J}_1,\; \text{O'N},\; \text{Ly}, \text{ or } {}^2\text{B}_2(2^{2m+1}) \text{ for some } m \geq 1, 
\]
i.e. all other simple groups $S$ with $s(S) \geq 4$ have $|\pi(S)| > 8$. 

If $S = \text{PSL}_3(4)$, then $\Gamma(S)$ consists of the four isolated vertices $2$, $3$, $5$, and $7$. 
Because $|\text{Out}(S)| = 12$ divides $|S|$, there must be $|\pi(H) \setminus \pi(S)| = 4$ primes dividing $|F(H)|$ but not $|S|$. 
These primes form a clique in $\Gamma(H)$, and they are all adjacent to $2$ by Lemma~\ref{lemma:obs}(iii), yielding a clique of size $5$. 
However, $\Gamma(H) \cong \Gamma(\text{Fi}_{23})$ has no such clique. 

If $S = \text{M}_{22}$, then $\Gamma(S)$ consists of the three isolated vertices $5$, $7$, and $11$, and the edge $\{2,3\}$. 
Because $|\text{Out}(S)| = 2$ divides $|S|$, there must be $|\pi(H) \setminus \pi(S)| = 3$ primes dividing $|F(H)|$ but not $|S|$. 
These primes form a clique in $\Gamma(H)$, and they are all adjacent to $2$ by Lemma~\ref{lemma:obs}(iii). 
Lemma~\ref{lemma:modules} implies that they are also all adjacent to $3$, because elements of order $3$ fix a non-zero vector in every representation of $S$ in characteristic coprime to $|S|$. 
Therefore, $\Gamma(H)$ contains a clique of size $5$, a contradiction (as above). 

If $S = \text{J}_1$, then $\Gamma(S)$ consists of the three isolated vertices $7$, $11$, and $19$, and the clique $\{2,3,5\}$. 
Because $\text{Out}(S)$ is trivial, there must be $|\pi(H) \setminus \pi(S)| = 2$ primes dividing $|F(H)|$ but not $|S|$. 
These primes form an edge in $\Gamma(H)$, and both are adjacent to $2$ by Lemma~\ref{lemma:obs}(iii). 
Given that $\Gamma(H)$ must have two isolated vertices, Lemma~\ref{lemma:modules} implies that $S$ must admit a faithful irreducible module $V$ in characteristic coprime to $|S|$ such that elements of order $p$ fix a non-zero vector in $V$ for at most one prime $p \in \{7,11,19\}$. 
The character table of $S$ shows that no such module exists. 

If $S = \text{O'N}$, then $\Gamma(S)$ is shown in Section~\ref{ss:O'N}. 
Because $|\text{Out}(S)| = 2$ divides $|S|$, there must be $|\pi(H) \setminus \pi(S)| = 1$ prime $r$ dividing $|F(H)|$ but not $|S|$, and $r$ is adjacent to $2$ in $\Gamma(H)$ by Lemma~\ref{lemma:obs}(iii). 
The character table of $S$ shows that every element of odd prime order in $S$ fixes a non-zero vector in every faithful irreducible representation of $S$ in characteristic $r$, so Lemma~\ref{lemma:modules} implies that $\Gamma(H)$ is connected, a contradiction. 

If $S = \text{Ly}$, then $\Gamma(S)$ is shown in Section~\ref{ss:Ly}. 
Note that $|\pi(S)| = |\pi(\text{Fi}_{23})|$, and that $\Gamma(\text{Fi}_{23})$ has exactly two more edges than $\Gamma(S)$. 
Moreover, $\text{Out}(S)$ is trivial, so $\pi(F(H))$ must be a subset of $\pi(S)$. 
Recall that $\Gamma(\text{Fi}_{23})$ contains a unique clique of size $4$ (see observation~(O1) above), while $\Gamma(S)$ has no such clique. 
In order to create a unique clique of size $4$ in $\Gamma(H)$, we must add exactly one of the edges $\{5,7\}$, $\{5,11\}$, or $\{7,11\}$ to $\Gamma(S)$. 
In particular, $|F(H)|$ must be divisible by some $r \in \{5,7,11\}$. 
By Lemma~\ref{lemma:modules}, there must exist a module for $S$ in characteristic $r$ such that elements of order $p$ fix a non-zero vector for exactly one $p \in \{5,7,11\} \setminus \{r\}$. 
If $r \in \{7,11\}$, then we can see that no such module exists by inspecting the $r$-modular character table of $S$, which is available in {\sf GAP}. 
The $5$-modular character table of $S$ is not available, but the $5$-modular character table of its maximal subgroup $3.\text{McL}.2$ shows that no such module exists for $r=5$ either.

If $S = {}^2\text{E}_6(2)$, then $\Gamma(S)$ is shown in Section~\ref{ss:Fi24}. 
Note that $|\pi(S)| = |\pi(\text{Fi}_{23})|$, and that $13$, $17$, and $19$ are (the only) isolated vertices in $\Gamma(S)$. 
Note that $\text{Out}(S) \cong \text{Sym}_3$. 
If $\overline{H} = H/F(H)$ contains an outer automorphism of $S$ of order $p \in \{2,3\}$, then $p$ is adjacent to both $13$ and $17$ in $\Gamma(\overline{H})$ and hence in $\Gamma(H)$, so $\Gamma(H)$ has at most one isolated vertex. 
Therefore, $F(H)$ must be non-trivial, and $F(H)$ must be a subset of $\pi(S)$ because $|\pi(S)| = |\pi(\text{Fi}_{23})|$. 
To extend $\Gamma(S)$ to a graph $\Gamma(H)$ isomorphic to $\Gamma(\text{Fi}_{23})$, we must add exactly one edge, which must be of the form $\{p,q\}$ where $p \in \{2,3\}$ and $q \in \{13,17,19\}$. 
In particular, $|F(H)|$ must be divisible by some $r \in \{2,3,13,17,19\}$. 
If $r \in \{13,17,19\}$, then Lemma~\ref{lemma:obs}(iii) yields the edge $\{2,r\}$, so Lemma~\ref{lemma:modules} implies that there must exist an $S$-module in characteristic $r$ on which all elements of odd prime order fix no non-zero vector. 
The $r$-modular character tables of $S$, which are available in {\sf GAP}, show that no such module exists. 
For $r \in \{2,3\}$, we inspect the $r$-modular character tables of the maximal subgroups $\text{Fi}_{22}$ and $\text{F}_4(2)$ of $S$, which respectively show that elements of orders $13$ and $17$ always fix non-zero vectors. 
This implies that $\Gamma(H)$ contains both edges $\{r,13\}$ and $\{r,17\}$, and therefore contains at least one more edge than $\Gamma(\text{Fi}_{23})$, a contradiction. 

Now suppose that $S = {}^2\text{B}_2(2^{2m+1})$ for some $m \geq 1$.
By Lemma~\ref{lemma:suz}, $\Gamma(S)$ consists of the isolated vertex $2$ and three cliques $C_1$, $C_2$, and $C_3$. 
Given that $\Gamma(\text{Fi}_{23})$ has $8$ vertices, we have $3 \leq N \leq 7$, where $N$ is the total number of vertices in $C_1$, $C_2$, and $C_3$. 
If $N=7$, then $\pi(H) = \pi(S)$ and exactly one of the $C_i$ must be an isolated vertex because $\Gamma(\text{Fi}_{23})$ has two isolated vertices and no cliques of size $5$ by Lemma~\ref{lemma:obs}(iii). 
The other two $C_i$ must have sizes $2$ and $4$, or $3$ and $3$. 
These cases are ruled out by observations (O1) and (O2) above, respectively. 
If $N=6$, then there is one prime $r \in \pi(H) \setminus \pi(S)$, and $\{2,r\}$ is an edge in $\Gamma(H)$ by Lemma~\ref{lemma:obs}(iii). 
Two of the $C_i$ must therefore be isolated vertices, and the third $C_i$ must be a clique of size $4$ sharing no vertex with the edge $\{2,r\}$, contradicting (O1). 
If $N=5$, then there are two primes $r_1, r_2 \in \pi(H) \setminus \pi(S)$, both adjacent to $2$ in $\Gamma(H)$. 
Therefore, $C_1$ and $C_2$ (say) must be isolated vertices, and $C_3$ must be a triangle. 
If $r_1$ and $r_2$ both divide $|\text{Out}(S)|$, then, because $\text{Out}(S)$ is cyclic, $\{2,r_1,r_2\}$ is a triangle disjoint from $C_3$, contradicting (O2). 
Therefore, at least one of the $r_i$ must divide $|F(H)|$ but not $|S|$. 
Lemmas~\ref{lemma:suz_reps} and~\ref{lemma:modules} then imply that every vertex in $\Gamma(H)$ is adjacent to this $r_i$, so $\Gamma(H)$ is connected and therefore not isomorphic to $\Gamma(\text{Fi}_{23})$. 
If $N=4$, then there are three primes $r_1, r_2,r_3 \in \pi(H) \setminus \pi(S)$, all adjacent to $2$, so $C_1$ and $C_2$ must be isolated vertices and $C_3$ must be an edge. 
At least one $r_i$ must divide $|F(H)|$, because if all of them divide $|\text{Out}(S)|$ then $\{2,r_1,r_2,r_3\}$ is a clique disjoint from the edge $C_3$, contradicting (O1). 
Lemmas~\ref{lemma:suz_reps} and~\ref{lemma:modules} then imply that $\Gamma(H)$ is connected, a contradiction. 
If $N=3$, then all of the $C_i$ are isolated vertices, so $m\in\{1,2\}$ by Lemma~\ref{lemma:suz}. 
There are now four primes $r_1,r_2,r_3,r_4 \in \pi(H) \setminus \pi(S)$, all adjacent to $2$. 
None of these primes can divide $|F(H)|$, or else $\Gamma(H)$ would be connected, so $|\text{Out}(S)|=2m+1 \in \{3,5\}$ must be divisible by three distinct primes, but it is not.

\subsubsection{{\bf The sub-case $s(S) = 3$ for $G = \textnormal{Fi}_{23}$}} \label{ss:Fi23s=3}
We now rule out the candidates for $S$ with $s(S) = 3$. 
By Theorem~\ref{thm:tables}, we must consider all $S$ in Table~\ref{tab:s(S)=3} with $|\pi(S)| \leq |\pi(\text{Fi}_{23})| = 8$ (other than $\text{Fi}_{23}$ itself). 
Because $\Gamma(\text{Fi}_{23})$ contains a coclique of size $5$, Lemma~\ref{lemma:obs}(ii) further implies that $\Gamma(S)$ must contain a coclique of size $4$, i.e. that $t(S) \geq 4$. 

We start with the sporadic groups, noting that we can ignore $\mathbb{B}$ because its prime graph has more than $8$ vertices.
If $S = \textnormal{HS}$, $\textnormal{J}_{3}$, or $\text{M}_{11}$, then $\Gamma(S)$ consists of three cliques (cf. Sections~\ref{ss:M11} and~\ref{ss:HS-J3}), so $t(S)=3<4$. 
The cases $S = \text{Suz}$, $\text{M}_{24}$, $\text{M}_{23}$, and $\text{Co}_2$ are ruled out by exactly the same arguments as in Section~\ref{ss:Ths=3}. 

If $S = \text{Th}$, then the prime graph of $S$ is shown in Section~\ref{ss:Th}. 
Because $\text{Out}(S)$ is trivial, there must be $|\pi(H) \setminus \pi(S)| = 1$ prime $r$ dividing $|F(H)|$ but not $|S|$. 
By Lemma~\ref{lemma:obs}(iii), $\{2,r\}$ is an edge in $\Gamma(H)$. 
The character table of $S$ shows that every element of order $19$ in $S$ fixes a non-zero vector in every faithful irreducible representation of $S$ in characteristic $r$, so Lemma~\ref{lemma:modules} implies that $\Gamma(H)$ contains the edge $\{19,r\}$. 
In particular, $\Gamma(H)$ does not have two isolated vertices, so $\Gamma(H)$ is not isomorphic to $\Gamma(\text{Fi}_{23})$. 

It remains to deal with the non-sporadic groups in Table~\ref{tab:s(S)=3}. 
We can ignore the groups $\text{E}_7(2)$ and $\text{E}_7(3)$ because their prime graphs have more than $|\pi(\text{Fi}_{23})| = 8$ vertices. 
All cases in which $S = \text{PSL}_2(q)$ or $\text{G}_2(3^k)$ are ruled out because $t(S)=3<4$, by Lemmas~\ref{lemma:A1} and~\ref{lemma:G2}, and the case $S = \text{PSU}_6(2)$ is ruled out for the same reason. 

Now suppose that $S = \text{Alt}_p$ for some prime $p>6$ such that $p-2$ is also prime. 
The constraint $|\pi(S)| \leq 8$ forces $p \in \{7,13,19\}$. 
The prime graph of $\text{Alt}_{19}$ has three vertices of degree $5$ (namely $2$, $3$, and $5$), whereas $\Gamma(\text{Fi}_{23})$ has only one such vertex (namely $2$), so $p \neq 19$. 
If $p \in \{7,13\}$, then $\Gamma(S)$ consists of three cliques, so $t(S)=3 < 4$. 

Now suppose that $S = \text{P}\Omega_{2p}^-(3)$ with $p = 2^m+1$ a prime, $m \geq 1$. 
We need $|\pi(S)| \leq 8$, so Lemma~\ref{lemma:2Dp} implies that $p \leq 7$, and hence $p \in \{3,5\}$. 
If $p=3$, then $S \cong \text{PSU}_4(3)$ and $\Gamma(S)$ consists of three cliques, so $t(S)=3<4$. 
If $p=5$, then $\Gamma(S)$ is shown in Section~\ref{ss:Bs=3}.
Because $|\text{Out}(S)|=4$, there must be $|\pi(H) \setminus \pi(S)|=1$ prime $r$ dividing $|F(H)|$ but not $|S|$, and $\{2,r\}$ must be an edge in $\Gamma(H)$. 
The character table of $S$ shows, in particular, that every element of order $41$ in $S$ fixes a non-zero vector in every faithful irreducible representation of $S$ in characteristic coprime to $|S|$, so Lemma~\ref{lemma:modules} implies that $\Gamma(H)$ also contains the edge $\{41,r\}$. 
In particular, $\Gamma(H)$ does not have two isolated vertices. 

If $S = \text{F}_4(q)$ with $q$ even, then Lemma~\ref{lemma:F4}(i) implies that $q \in \{2,4\}$. 
Lemma~\ref{lemma:F4}(ii) rules out $q=4$ because $\Gamma(\text{Fi}_{23})$ has only $9$ edges. 
The case $q=2$ is ruled out by exactly the same argument as in Section~\ref{ss:Ths=3}.

If $S = {}^2\text{F}_4(2^{2m+1})$, $m \geq 2$, then Lemma~\ref{lemma:2F4} implies that $\Gamma(S)$ contains a subgraph that does not arise in $\Gamma(\text{Fi}_{23})$. 
If $S = {}^2\text{F}_4(8)$, then $\Gamma(S)$ is shown in the proof of Lemma~\ref{lemma:2F4}. 
By inspection, $\Gamma(\text{Fi}_{23})$ has no subgraph isomorphic to $\Gamma(S)$, so this case is ruled out too. 

Finally, suppose that $S = {}^2\text{G}_2(3^{2m+1})$, $m \geq 1$. 
The case $m=1$ is ruled out by the same argument as in Section~\ref{ss:Ths=3}. 
If $m \geq 2$, then because $\Gamma(S)$ must have two isolated vertices (by Lemma~\ref{lemma:obs}(iii)), Lemma~\ref{lemma:2G2_2} implies that $|\pi(S)| \geq 9 > |\pi(\text{Fi}_{23})|$, a contradiction.

\end{document}